\documentclass{amsart}
\usepackage{amssymb}
\usepackage{amsmath}
\usepackage{amsfonts}

\begin{document}

\newtheorem{theo}{Theorem}[section]
\newtheorem{atheo}{Theorem*}
\newtheorem{prop}[theo]{Proposition}
\newtheorem{aprop}[atheo]{Proposition*}
\newtheorem{lemma}[theo]{Lemma}
\newtheorem{alemma}[atheo]{Lemma*}
\newtheorem{exam}[theo]{Example}
\newtheorem{coro}[theo]{Corollary}
\theoremstyle{definition}
\newtheorem{defi}[theo]{Definition}
\newtheorem{rem}[theo]{Remark}

\title[Lipschitz Bernoulli utility functions]
{Lipschitz Bernoulli utility functions}
\author{Efe A. Ok}
\address{Department of Economics and Courant Institute of Mathematical
Sciences, New York University}
\email{efe.ok@nyu.edu}
\author{Nik Weaver}
\address{Department of Mathematics, Washington University in St. Louis}
\email{nweaver@wustl.edu}
\date{April 22, 2021}
\subjclass[2020]{Primary 46N10, 91B06; Secondary 06A06, 46E15}
\keywords{Bernoulli utility, Lipschitz functions, Lipschitz preorders, expected utility
representation, Kantorovich-Rubinstein space, Wasserstein metric}
%\thanks{This paper is in final form and no version of it will be submitted for publication elsewhere.}

%%%%%%%%%%%%%%%%%%%%%%%%%%%%%%%%%%%%%%%%%%%%%%%%%%%%%%%%%%%%%%%%%%%%%%%
%   Please insert the article body 
%%%%%%%%%%%%%%%%%%%%%%%%%%%%%%%%%%%%%%%%%%%%%%%%%%%%%%%%%%%%%%%%%%%%%%%

\begin{abstract}
We obtain variants of the classical von Neumann-Morgenstern expected
utility theorem, with and without the completeness axiom, in which the
derived Bernoulli utility functions are Lipschitz. The prize space in these
results is an arbitrary separable metric space, and the utility functions
may be unbounded. The main ingredient of our results is a novel
(behavioral) axiom on the underlying preference relations which is satisfied
by virtually all stochastic orders. The proof of the main representation theorem
is built on the fact that the completion of the Kantorovich-Rubinstein space is
the canonical predual of the Banach space of Lipschitz functions that vanish
at a fixed point. Two applications are given, one to the theory of
non-expected utility theory, and the other to the theory of decision-making
under uncertainty.
\end{abstract}

\maketitle

\section{Introduction}

One of the fundamental theorems of modern decision theory is the von
Neumann-Morgenstern {\it expected utility theorem}. The most common
version of this result considers a total preorder $\succsim$ on the space
$\Delta (X)$ of all Borel probability measures on a separable metric space $X$.
Here one interprets $\succsim$ as the preference relation of an
individual over a collection of risky prospects/lotteries. The theorem says
that if $\succsim$ is affine (i.e., $p\succsim q$ iff $(1-\lambda)p +
\lambda r\succsim (1-\lambda) q+\lambda r$, for any $p,q,r\in \Delta (X)$
and $\lambda \in [0,1)$) and is closed relative to the topology of
weak convergence, then there exists a continuous and bounded ``utility''
function $u: X \to \mathbb{R}$ such that $p\succsim q$ iff
$\int u\, {\rm d}p \geq \int u\, {\rm d}q$, i.e., the
expectation of $u$ with respect to $p$ is at least as large as that with
respect to $q$. So, on the basis of simple rationality axioms (transitivity
and affinity), and basic regularity conditions (totalness and continuity),
one arrives at the principal idea behind Daniel Bernoulli's famous
resolution of the St. Petersburg paradox: a rational individual would
evaluate lotteries on the basis of their expected utility. The function $u$
found in the said theorem is thus commonly referred to as a
{\it Bernoulli utility function}.

It is argued by many authors that requiring $\succsim$ to be total is
unnecessarily demanding, for rational individuals may well be unable to
rank even two riskless alternatives (say, because they use multiple,
potentially conflicting criteria in their evaluation, or because they have
insufficient information about the alternatives). Fortunately, when $X$ is
compact, relaxing this assumption in the von Neumann-Morgenstern theorem
alters its statement in a tractable fashion. In that case, for any given
continuous and affine, but not necessarily total, preorder $\succsim$ on
$\Delta (X)$ one obtains a  {\it set} $\mathcal{U}$ of continuous (Bernoulli
utility) functions such that
\begin{equation}\label{MU}
p\succsim q\qquad \mbox{iff}\qquad \int_X u\, {\rm d}p \geq
\int_X u\, {\rm d}q\quad\mbox{for every }u\in \mathcal{U}.
\end{equation}
This result is known as the {\it expected multi-utility theorem}.
(Preliminaries that we need from expected utility theory are reviewed in
Section 3.)

As basic as they are, there are two issues with these results. First,
neither of them says anything about the structure of Bernoulli utilities, other
than that they are continuous and bounded. In particular, in applications
where $X = \mathbb{R}$ (in which case we talk of monetary lotteries), it is
common to work with differentiable, sometimes even analytic, Bernoulli
utility functions, but it is not clear what sort of a behavioral
property (to be imposed on $\succsim$) would actually ensure the
differentiability of a Bernoulli utility. Second, it is a bit limiting that
these results force Bernoulli utilities to be bounded. For instance, most models
in economics posit that individuals are risk averse in the sense that
they would prefer a degenerate lottery that pays $\$x$ for sure to any
monetary lottery whose expected value is $x$. But the preferences of such
an individual cannot be represented as envisaged by the von
Neumann-Morgenstern theorem. If they did, the Bernoulli utility of the agent
would be a concave map on $\mathbb{R}$ (by Jensen's inequality), but no such
map is bounded. The expected multi-utility theorem is even more restrictive
in this regard, for the compactness hypothesis in that result is essential:
even when $X = \mathbb{R},$ the assumptions of the theorem do not yield
continuous Bernoulli utilities. Thus, this result does not apply to
preferences that are able to rank normally, or exponentially, distributed
random variables, which is obviously a significant shortcoming.

The culprit behind these limitations is the requirement that $\succsim$ be a closed
subset of $\Delta (X)\times \Delta (X)$. Because it does not reference the metric on $X$,
this assumption is too weak to give Bernoulli utilities any extra structure beyond
continuity, and because it
requires $\succsim$ to be defined on all of $\Delta(X)$, which includes badly behaved
measures when $X$ is unbounded, it is too strong to accomodate unbounded utilities. In
contrast, the present paper is built on a new continuity notion that leads to expected
utility theorems which suffer from these impediments to a
lesser extent. To start with, we assume that $X$ is separable throughout the
analysis, and restrict our attention to $\Delta _1(X)$, the set of all
Borel probability measures relative to which any Lipschitz function on $X$
is integrable.\footnote{Here and throughout the paper, by ``Lipschitz
function'' we always mean ``real-valued Lipschitz function.''} This allows
us to endow $\Delta_1(X)$ with the Wasserstein 1-metric $W_1$ in lieu of the
topology of weak convergence. While essential for the present approach, working
on $\Delta _1(X)$ instead of $\Delta (X)$ is not really restrictive. For
instance, $\Delta _1(\mathbb{R})$ is just the set of all Borel probability
measures on $\mathbb{R}$ with finite mean, and $\Delta _1(X) = \Delta (X)$
if $X$ is bounded.

The main innovation of our work is to replace the classical continuity property of
a preorder $\succsim $ on $\Delta_1(X)$ with a requirement of the following form:
if $q$ does not have equivalent or superior value to $p$, then for any $p'$ and $q'$
the mixture $(1-\lambda )q+\lambda q'$ will not have equivalent or superior value to
$(1-\lambda)p+\lambda p'$, provided $\lambda \geq 0$ is smaller
than a critical value that depends on the distance between $p'$ and $q'$. For
reasons that we discuss in Section 4.1, we posit this critical value to have
the form $\frac{K}{K+W_{1}(p', q')}$ for some $K > 0$
(which depends only on $p$ and $q$). Thus, we designate a
preorder $\succsim$ on $\Delta_1(X)$ as {\it Lipschitz} if for every
$p,q\in \Delta_1(X)$ such that $q\succsim p$ fails, there is a $K > 0$ such that
$(1 - \lambda )q+\lambda q'\succsim (1-\lambda )p+\lambda p'$ fails for all
$p',q' \in \Delta_1(X)$ and all $\lambda $ in $\left[0,\frac{K}{K+W_{1}(p', q')}\right)$.
There are plenty of interesting
examples of Lipschitz preorders. In particular, we show in Section 4.2 that
virtually any stochastic order is indeed affine and Lipschitz in this sense.

Lipschitz-type properties have occasionally appeared in the literature on decision theory. The
one most directly comparable to ours is due to Levin \cite{levin}. Levin's Lipschitz condition is very
different from ours in both spirit and formalism, yet we prove in Section 4.7 that the two
definitions are equivalent for affine preorders, befitting the terminology we adopt here.

The main result of the present paper is presented in Section 4.3.
We prove that a preorder $\succsim $ on $\Delta_1(X)$ is Lipschitz
and affine if and only if there is a nonempty family $\mathcal{U}$ of
Lipschitz (Bernoulli utility) functions such that (\ref{MU}) holds. The fact
that any preorder that can be represented this way is Lipschitz uses the
Kantorovich-Rubinstein duality theorem (reviewed below in Section 2.1).
To prove the converse, we associate to
$\succsim$ the positive cone $C:=\{\alpha (p-q):\alpha \geq 0$ and $p\succsim
q\}$, which sits in the Kantorovich-Rubinstein space ${\rm KR}(X)$.
Affinity of $\succsim$ entails that $C$ is convex and that $p\succsim q$ iff
$p - q\in C.$ We then use the Lipschitz property of $\succsim$ to show that $C$
is closed in ${\rm KR}(X)$. Thus, $C$ equals the intersection of all
closed halfspaces in ${\rm KR}(X)$ that contain it. Exploiting the fact
that the dual of ${\rm KR}(X)$ is the Banach space ${\rm Lip}_0(X)$ of all
Lipschitz functions on $X$ that vanish at a fixed point (with the Lipschitz
number acting as the norm) -- see Section 2.3 -- we obtain the desired
characterization.\footnote{The expected multi-utility theorem was proved in
\cite{dmo} by an analogous method, but utilizing instead the fact that
$C(X)^* \cong M(X)$ when $X$ is compact.} We then use this duality in Section 4.4
to show that $\mathcal{U}$ can be chosen in our main result as weak*-compact
and convex, and to identify in what way we can think of the $\mathcal{U}$ in our
characterization as unique. Then in Section 4.5 we show that one can
guarantee the set $\mathcal{U}$ in (\ref{MU}) to contain only strictly
$\succsim$-increasing functions,
but that comes at the cost of losing its weak*-compactness. In
Section 5 we discuss how one may use these alternate representations of a
Lipschitz affine preorder $\succsim$ on $\Delta_1(X)$ to find the
$\succsim$-maximal lotteries in any set $P\subseteq \Delta_1(X)$.

As a special case of our main representation theorem, we find that a total
Lipschitz affine preorder $\succsim$ on $\Delta_1(X)$ admits an expected
utility representation exactly as in the classical von Neumann-Morgenstern
theorem, but now with a Lipschitz (and possibly unbounded) Bernoulli utility
function. Moreover, the duality between direct sums of finitely many copies of
${\rm KR}(X)^*$ and ${\rm Lip}_0(X)$ yields another type of expected utility
theorem known as a state-dependent expected multi-utility theorem (Section
7.1). In each of these results, $X$ remains a separable metric space, and
the derived Bernoulli utilities are allowed to be unbounded.

There are several advantages to working with Lipschitz Bernoulli utility
functions, as opposed to those that are merely continuous. For instance, in
most applications of decision theory, the prize space is taken to be a measurable
subset of $\mathbb{R}^n$, and in that context any
Lipschitz utility is almost everywhere differentiable. Moreover, the maxima
of such utility functions can be studied by means of nonsmooth analysis
(Section 3.5). Of course, insofar as applications are concerned, one can
simply assume that an individual has a Lipschitz Bernoulli utility. Instead,
the ``use'' of our findings is foundational. They translate the seemingly
technical assumption of Lipschitz
continuity into the behavioral language of preference relations.

We conclude the paper with two applications of our representation theorems.
The first one utilizes our main theorem and the equivalence of our Lipschitz
condition with that of Levin for affine preorders. Put precisely, we prove
that if $\succsim$ is a Lipschitz preorder on $\Delta_1(X)$ in the sense
of Levin, then its affine core, that is, the largest affine preorder
contained in $\succsim$, exists and admits an expected multi-utility
representation with Lipschitz Bernoulli utilities. This seems like a useful
observation because the affine cores of preorders are routinely used in
non-expected utility theory (where $\succsim$ need not be affine). In our
second application, we consider preferences over Anscombe-Aumann acts (i.e.,
state-dependent lotteries), and prove a subjective expected utility theorem
in which the beliefs of a decision-maker over the states of nature are
uniquely identified while their preferences over riskless prizes are again
captured by a set of Lipschitz Bernoulli utilities.

\section{Preliminaries on Kantorovich-Rubinstein theory}

\subsection{$W_1$-metrization of Borel probability measures}
We denote the set of all Borel probability measures on a
metric space $X = (X, d)$ by $\Delta(X)$. We will primarily work with the following
subset of $\Delta(X)$:
\begin{equation}
\Delta_1(X) := \left\{p \in \Delta(X): \int_{X\times X} d\, {\rm d}(p\times p) < \infty\right\}.
\end{equation}
We can think of the members of $\Delta_1(X)$ as those
probability measures on $X$ whose spreads are controlled by the metric $d$.
This is made clearer by the following result.

\begin{prop}\label{altdelta1}
Let $X = (X,d)$ be a metric space and let $\mu$ be a finite positive
Borel measure on $X$. Then the following are equivalent:
\smallskip

{\narrower{
\noindent (a) $\int d\, {\rm d}(\mu\times \mu) < \infty$;
\smallskip

\noindent (b) $\int d(\cdot, e)\, {\rm d}\mu < \infty$ for some (any) $e \in X$;
\smallskip

\noindent (c) every Lipschitz function on $X$ is integrable against $\mu$.

\smallskip}}
\end{prop}

\begin{proof}
Fix $e \in X$ arbitrarily and use Tonelli's theorem to observe that
\begin{eqnarray*}
\int_{X\times X} d\, {\rm d}(\mu\times \mu)
&=& \int_X\int_X d(x,y)\, \mu({\rm d}x)\, \mu({\rm d}y)\cr
&\leq& \int_X\int_X (d(x,e) + d(y,e))\, \mu({\rm d}x)\, \mu({\rm d}y)\cr
&=& \mu(X)\int_X d(x,e)\, \mu({\rm d}x) + \mu(X)\int_X d(y,e)\, \mu({\rm d}y)\cr
&=& 2\mu(X)\int_X d(\cdot, e)\, {\rm d}\mu.
\end{eqnarray*}
Thus, if every Lipschitz function is integrable against $\mu$, then in particular
$d(\cdot, e)$ is integrable against $\mu$, and we conclude from the above
that $\int d\, {\rm d}(\mu\times \mu) < \infty$. This establishes
(c) $\Rightarrow$ (b) $\Rightarrow$ (a). For (a) $\Rightarrow$
(c), suppose some Lipschitz function $f$ is not integrable against
$\mu$. By taking the positive or negative part of $f$, without loss of generality
we may assume $f \geq 0$. Now $f(\cdot) \leq f(e) + Ld(\cdot, e)$ where
$L$ is the Lipschitz number of $f$ and $e \in X$ is still arbitrary, and
$\int f(e)\, {\rm d}\mu = f(e)\mu(X)$ is finite, so
$\int f\, {\rm d}\mu = \infty$ implies
$\int d(\cdot,e)\, {\rm d}\mu = \infty$. Since $e$ is arbitrary, it follows that
\begin{equation*}
\int_X\int_X d(x,y)\, \mu({\rm d}x)\, \mu({\rm d}y) =
\int_X \infty\cdot \mu({\rm d}y) = \infty.
\end{equation*}
This shows that if some Lipschitz function were not integrable against $\mu$ then
condition (a) would fail.
\end{proof}

One easy consequence of this result is that $\Delta_{1}(X) = \Delta (X)$ iff $X$ is bounded. Indeed,
if $X$ is unbounded then we can fix any $e \in X$ and find a sequence $(x_n)$ with
$d(x_n, e) \geq 2^n$ for all $n$, and then $d(\cdot, e)$ will not be integrable against
the probability measure $\sum 2^{-n}\delta_{x_n}$.

The {\it Wasserstein 1-metric} $W_1$ on $\Delta_1(X)$ is defined as
\begin{equation}
W_1(p,q) := \inf \int_{X\times X} d\, {\rm d}\mu,
\end{equation}
taking the infimum over all couplings of $p$ and $q$, that is, all Borel
probability measures $\mu$ on $X\times X$ with marginals $p$ and $q$ on
the first and second components, respectively.\footnote{The results
which follow are standard. They are proved in, say, \cite[Chapter 21]{garling} or
\cite[Chapter 7]{villani}. The original form of the
Kantorovich-Rubinstein theorem assumes that $X$ is compact; it was
extended to all separable metric spaces in \cite{dudley1} and \cite{acosta}.
The form of the result used here is presented as Theorem 11.8.2 in \cite{dudley2}.}
It is well-known that $W_1$ is indeed a metric on $\Delta_1(X)$, and that the
infimum in its definition is attained if $X$ is a Polish metric space. It is also worth
noting that convergence with respect to $W_1$ implies weak convergence;
in fact, for any $p, p_1, p_2, \ldots \in \Delta_1(X)$, we have $W_1(p_n,p) \to 0$
iff $\{p_0, p_1, \ldots\}$ is uniformly integrable and $(p_n)$ converges to
$p$ weakly. (Here ``uniformly integrable'' means that for some, hence any,
$e \in X$
\begin{equation*}
\limsup_m \int_{X_K} d(\cdot, e)\, {\rm d}p_m \to 0
\end{equation*}
as $K \to \infty$, where $X_K := \{x \in X: d(x,e) \geq K\}$.)

The Kantorovich-Rubinstein duality theorem states that if $X$ is separable
then for any $p,q \in \Delta_1(X)$
\begin{equation}\label{krthm}
W_1(p,q) = \sup \left|\int_X f\, {\rm d}(p-q)\right|,
\end{equation}
taking the supremum over all 1-Lipschitz functions $f$ on $X$.
This is the characterization of $W_1$ that is relevant to our purposes, so
separability will be assumed throughout the paper.

\subsection{The Kantorovich-Rubinstein space}
The Kantorovich-Rubinstein space ${\rm KR}(X)$ has been studied for compact
metric spaces $X$ (Section VIII.4 of \cite{kr}; see also Section 2.3 in the first edition
of \cite{weaver}). As we are not aware of any systematic treatment of the noncompact
case, we include proofs of basic facts about ${\rm KR}(X)$ in that setting here.
They are all straightforward generalizations from the compact setting.

Let $X = (X,d)$ be a separable metric space. We define ${\rm KR}(X)$ to be the set
of all finite signed Borel measures $\mu$ on $X$ satisfying $\mu(X) = 0$ and
\begin{equation}
\int_{X\times X} d\, {\rm d}(|\mu|\times |\mu|) < \infty.
\end{equation}
By Proposition \ref{altdelta1}, $\mu$ belongs to ${\rm KR}(X)$ iff
$\mu(X) = 0$ and $\int d(\cdot, e)\, {\rm d}|\mu| < \infty$ for some (any)
$e \in X$. This characterization plus the fact that $|\mu + \nu| \leq |\mu| + |\nu|$
for any signed measures $\mu$ and $\nu$ shows that ${\rm KR}(X)$ is a real vector space.

We equip ${\rm KR}(X)$ with the norm
\begin{equation}
\|\mu\|_{\rm KR} := \sup \left|\int_X f\, {\rm d}\mu\right|,
\end{equation}
taking the supremum over all 1-Lipschitz functions $f$ on $X$.
This supremum is finite because if $f$ is 1-Lipschitz then
$|f(\cdot) - f(e)| \leq d(\cdot, e)$ and so (using $\mu(X) = 0$)
\begin{equation*}
\left|\int_X f\, {\rm d}\mu\right| = \left|\int_X (f - f(e))\, {\rm d}\mu\right|
\leq \int_X d(\cdot, e)\, {\rm d}|\mu| < \infty.
\end{equation*}
It is clear that $\|\cdot\|_{\rm KR}$ is a seminorm. We will next prove that it is
actually a norm (albeit generally not a complete norm).\footnote{When the diameter of
$X$ is finite, ${\rm KR}(X)$ is complete if and only if $X$ is uniformly
discrete, i.e., there exists $a > 0$ such that $d(x,y) \geq a$ for every distinct $x,y \in X$.}

\begin{prop}\label{krprop}
Let $X$ be a separable metric space. Then
\begin{equation}
{\rm KR}(X) = \{\alpha(p - q): \alpha \geq 0\hbox{ and }p,q \in \Delta_1(X)\},
\end{equation}
with $\|\alpha(p - q)\|_{\rm KR} = \alpha W_1(p,q)$.
Moreover, $\|\cdot\|_{\rm KR}$ is a norm, and the set of finitely supported
measures in ${\rm KR}(X)$ is dense relative to this norm.
\end{prop}

\begin{proof}
As ${\rm KR}(X)$ is a vector space, $\alpha(p-q) \in {\rm KR}(X)$
whenever $\alpha \geq 0$ and $p,q \in \Delta_1(X)$. Conversely, for any nonzero
$\mu \in {\rm KR}(X)$ we have $\mu = \alpha(p - q)$ where $\alpha := \mu^+(X) = \mu^-(X)$,
$p := \frac{1}{\alpha}\mu^+$, and $q := \frac{1}{\alpha}\mu^-$. This proves the first assertion.
The second assertion follows from the simple fact that $\|\mu\|_{\rm KR} = \alpha\|p - q\|_{\rm KR}$,
together with the Kantorovich-Rubinstein duality theorem.

To see that $\|\cdot\|_{\rm KR}$ is a norm, not merely a seminorm, suppose $\|\mu\|_{\rm KR} = 0$,
i.e., $\int f\, {\rm d}\mu = 0$ for every Lipschitz function $f$.
Equivalently, $\int f\, {\rm d}\mu^+ = \int f\, {\rm d}\mu^-$ for every Lipschitz function
$f$. We must show that $\mu = 0$. To this end, fix a closed set $C \subseteq X$, and for
$n \in \mathbb{N}$ define
\begin{equation*}
f_n(x) := {\rm max}(1 - n\cdot d(x,C), 0).
\end{equation*}
This sequence converges boundedly pointwise to the characteristic function $1_C$, so by the
dominated convergence theorem $\int f_n\, {\rm d}\mu^+ = \int f_n\, {\rm d}\mu^-$ for all $n$
implies $\int 1_C\, {\rm d}\mu^+ = \int 1_C\, {\rm d}\mu^-$, i.e., $\mu^+(C) = \mu^-(C)$.
Since $C$ was arbitrary, we have shown that if $\mu^+$ and $\mu^-$ agree when integrated
against any Lipschitz function then they agree on all closed sets, and therefore they are equal,
which yields the desired conclusion that $\mu = 0$.

For the last assertion, let $\mu \in {\rm KR}(X)$, let $\epsilon > 0$, and let $(x_n)$ be a
dense sequence in $X$. Set $\epsilon' := \frac{\epsilon}{|\mu|(X) + 1}$ and
$U_1 := {\rm ball}_{\epsilon'}(x_1)$ (the open ball about $x_1$ of radius $\epsilon'$)
and inductively define
$U_{n+1} := {\rm ball}_{\epsilon'}(x_{n+1}) \setminus(U_1 \cup \cdots \cup U_n)$.
Fix $e \in X$. Then $X = \bigcup U_n$, and so there exists $N > 0$ such that
\begin{equation*}
\int_X d(\cdot, e)\, {\rm d}|\mu| \leq \int_{X_N} d(\cdot, e)\, {\rm d}|\mu| + \epsilon',
\end{equation*}
where $X_N := U_1 \cup \cdots \cup U_N$. Define
\begin{equation*}
\mu_0 := \sum_{n=1}^N \mu(U_n)\cdot \delta_{x_n} + \mu(X\setminus X_N)\cdot \delta_e.
\end{equation*}
Then $\mu_0$ is a finitely supported measure in
${\rm KR}(X)$, and for any 1-Lipshitz function $f$ on $X$ satisfying $f(e) = 0$ we have
\begin{equation}\label{krineq}
\left|\int_X f\, {\rm d}(\mu - \mu_0)\right| \leq \epsilon'(|\mu|(X) + 1) = \epsilon
\end{equation}
because
\begin{equation*}
\left|\int_{U_n} f\, {\rm d}(\mu - \mu_0)\right| \leq \epsilon'\cdot |\mu|(U_n)
\end{equation*}
for each $1 \leq n \leq N$ and
\begin{equation*}
\left|\int_{X\setminus X_N} f\, {\rm d}(\mu - \mu_0)\right|
= \left|\int_{X \setminus X_N} f\, {\rm d}\mu\right|
\leq \int_{X \setminus X_N} d(\cdot, e)\, {\rm d}|\mu|
\leq \epsilon'
\end{equation*}
(since $|f(\cdot)| \leq d(\cdot, e)$). Since the integral of any constant function
against any measure in ${\rm KR}(X)$ is zero, given any 1-Lipschitz function
$f$ on $X$ we can apply (\ref{krineq}) to $f - f(e)$ and obtain
$|\int f\, {\rm d}(\mu - \mu_0)| = |\int (f - f(e))\, {\rm d}(\mu - \mu_0)|
\leq \epsilon$,
showing that $\|\mu - \mu_0\|_{\rm KR} \leq \epsilon$.
This completes the proof.
\end{proof}

The idea of working with functions that vanish at an arbitrarily chosen ``base point''
also appears in the context of Lipschitz spaces. A metric space $X = (X, d, e)$ with a
specified base
point $e$ is called {\it pointed}. Given such a space we define ${\rm Lip}_0(X)$ to be the set
of all Lipschitz functions on $X$ which vanish at $e$. This becomes a Banach space if we
take the norm of a function to be its Lipschitz number. In fact, it is a dual space:
relative to the pairing
\begin{equation*}
\langle \mu, f\rangle := \int_X f\, {\rm d}\mu
\end{equation*}
($\mu \in {\rm KR}(X)$, $f \in {\rm Lip}_0(X)$), the canonical predual of ${\rm Lip}_0(X)$
is the completion of the space of finitely supported measures in ${\rm KR}(X)$
\cite[Theorem 3.3]{weaver}. As Proposition \ref{krprop} shows that ${\rm KR}(X)$ is
(densely) contained in this completion, the next result is
immediate.\footnote{The fact that ${\rm KR}(X)$ does not depend on the base point
shows that, up to isometric isomorphism, ${\rm Lip}_0(X)$ does not either.}

\begin{theo}\label{duality}
${\rm KR}(X)^* \cong {\rm Lip}_0(X)$ for any separable pointed metric space $X$.
\end{theo}

Thus, if $X$ is separable then every bounded linear functional on ${\rm KR}(X)$
is given by integration against a Lipschitz function on $X$. This fact will play an
essential role in what follows.

On bounded subsets of ${\rm Lip}_0(X)$, the weak* topology on ${\rm Lip}_0(X)$ agrees with the
topology of pointwise convergence. If $X$ is compact, then on bounded sets
this is the same as the topology of uniform convergence. Here ``bounded''
means ``bounded in norm (i.e., Lipschitz number) in ${\rm Lip}_0(X)$.'' Note
that ${\rm Lip}_0(X)$ does contain unbounded functions if the diameter of $X$ is
infinite.

In the sequel we will also need a minor generalization of Theorem \ref{duality}. Take
any nonempty finite index set $I$, and for each $i \in I$ let $X_i=(X_i,d_i,e_i)$ be
a separable pointed metric space. We consider the direct
sums
\begin{equation*}
\bigoplus_{i\in I} {\rm Lip}_0(X_i)\qquad\mbox{and}\qquad
\bigoplus_{i\in I} {\rm KR}(X_i)
\end{equation*}
as normed linear spaces relative to the norms
\begin{equation*}
\left\| (f_i)_{i\in I}\right\| :=\max_{i\in I}L(f_i)\qquad\mbox{and}\qquad
\left\| (\mu_i)_{i\in I}\right\|
:=\sum_{i\in I}\left\| \mu_i\right\|_{\rm KR},
\end{equation*}
where $L(f_i)$ stands for the Lipschitz number of $f_i$. Then, given
Theorem \ref{duality}, it is an easy exercise to prove that 
$\left( \bigoplus_{i\in I} {\rm KR}(X_i)\right)^* \cong
\bigoplus_{i\in I} {\rm Lip}_0(X_i)$, where the duality pairing is 
\begin{equation*}
\left\langle (\mu_i)_{i\in I}, (f_i)_{i\in I}\right\rangle
:=\sum_{i\in I}\int_X f_i\, {\rm d}\mu_i.
\end{equation*}
We will use this fact in the case where $X_i=X$ for each $i \in I$, where it entails that
for every bounded linear functional $\Phi$ on $\bigoplus_{i\in I}
{\rm KR}(X)$ there are Lipschitz maps $f_i$ on $X$ for each $i\in I$
such that $\Phi ((\mu _i)_{i\in I}) =\sum_{i\in I}\int f_i\, {\rm d}\mu _{i}$
for any $(\mu_i)_{i\in I}$ in $\bigoplus_{i\in I} {\rm KR}(X)$.

\section{Preliminaries on expected utility theory}

\subsection{Affine preorders}
By  a {\it preorder} on any nonempty set we mean a reflexive and transitive
binary relation $\succsim$ on that set, and we denote the symmetric and
antisymmetric parts of $\succsim$ by $\sim$ and $\succ$, respectively.
We say that a preorder $\succsim$ on a nonempty convex subset $S$ of
$\Delta(X)$ is {\it weakly affine} if
\begin{equation}
p \succsim q \qquad\mbox{implies} \qquad (1 - \lambda)p + \lambda r
\succsim (1 - \lambda)q + \lambda r
\end{equation}
for every $p, q, r \in S$ and $\lambda \in [0,1)$, and {\it affine} if
\begin{equation}
p \succsim q \qquad \mbox{iff} \qquad (1 - \lambda)p + \lambda r
\succsim (1 - \lambda)q + \lambda r
\end{equation}
for every $p,q,r \in S$ and $\lambda \in [0,1)$. It is easily checked using
the second part of Lemma \ref{basiclemma} below that $\succsim$ is weakly
affine iff it is a convex subset of $\Delta(X) \times \Delta(X)$. If $\succsim$
is affine, so are $\sim$ and $\succ$, but weak affinity of $\succsim$ ensures
only that $\sim$ is weakly affine.
For example, take any integer $n \geq 2$, let $X := \{1, \ldots, n\}$, and let $\succsim$
be the binary relation on $\Delta(X)$ defined by $p \succsim q$
iff ${\rm min}\{i \in X: p\{i\} > 0 \} \geq {\rm min}\{i \in X: q\{i\} > 0\}$. Then $\succsim$
is a weakly affine preorder, but it is not affine, nor is $\succ$ is not weakly affine.

The property of weak affinity was introduced in the 1947 opus of von Neumann
and Morgenstern \cite{vnm} as a fundamental
trait of rationality for one's preferences over risky
alternatives.\footnote{Consider tossing a coin whose probability of coming up
heads is $\lambda$. Now interpret $(1 - \lambda)p + \lambda r$ as the
compound lottery in which this coin is tossed and a reward (in $X$) is
obtained according to $p$ if the outcome of the toss is tails, and according
to $r$ if the outcome is heads. Interpreting $(1 - \lambda)q + \lambda r$
similarly, the affinity property simply says that one's preferences over the
comparison of these two compound lotteries should be consistent with
one's preferences over $p$ and $q$ (no matter what $\lambda$ is).}
This property has been studied extensively in the literature on mathematical
decision theory (where it is often referred to as the ``independence
axiom'').\footnote{See, among countless others, \cite{aumann, dmo,
grandmont, hor, hw, hm}.}
In applications it is extremely rare that one encounters a weakly affine but
not affine preorder. Moreover, the ``rationality'' motivation of weak affinity
also applies to affinity. (In fact, in Section 6.1 we shall demonstrate that there
is no difference between these properties under a mild continuity condition.)
As a result, most authors in this field either implicitly or explicitly work with
affine preorders.

We adopt the following notation:
\begin{equation}
p \underset{\lambda}{\oplus} q := (1 - \lambda)p + \lambda q
\end{equation}
for any $p,q \in \Delta(X)$ and $\lambda \in [0,1]$. Thus a preorder
$\succsim$ on $S$ is affine provided $p \succsim q$ iff $p \oplus_\lambda r \succsim
q \oplus_\lambda r$, for every $p,q,r \in S$ and $\lambda \in [0,1)$.

We note two easy facts about weakly affine preorders:

\begin{lemma}\label{basiclemma}
Let $X$ be a metric space and $\succsim$ a weakly affine preorder
on a nonempty convex subset $S$ of $\Delta(X)$. Then for any
$p, p', q, q', r \in S$ and any $\lambda \in [0,1]$
\begin{equation*}
p \succsim q\qquad \mbox{implies}\qquad
r \underset{\lambda}{\oplus} p \succsim r \underset{\lambda}{\oplus} q
\end{equation*}
and
\begin{equation*}
p \succsim q\mbox{ and }p' \succsim q'\qquad\mbox{implies}\qquad
p \underset{\lambda}{\oplus} p' \succsim q \underset{\lambda}{\oplus} q'.
\end{equation*}
\end{lemma}

\begin{proof}
The first statement is shown merely by replacing $\lambda$ by $1 - \lambda$,
and the second one follows because if $p \succsim q$ and $p'\succsim q'$, then,
using weak affinity twice, $p \oplus_\lambda p' \succsim q \oplus_\lambda p' \succsim
q \oplus_\lambda q'$.
\end{proof}

\subsection{Continuity of preorders}
For any separable metric space $X$, we say that a preorder $\succsim$ on a nonempty convex
subset $S$ of $\Delta(X)$ is {\it continuous} if it is a closed subset of $S \times S$ relative
to the product topology, using the topology of weak convergence in $S$ on each
factor. If $S \subseteq \Delta_1(X)$, we say that $\succsim$ is {\it $W_1$-continuous}
if it is a closed subset
of $S \times S$ relative to the product topology induced by the $W_1$ metric on $S$.
As convergence with respect to $W_1$ implies weak convergence, every continuous
preorder $\succsim$ on $S$ is $W_1$-continuous.

\subsection{The expected utility theorem}
For any metric space $X$, a total preorder $\succsim$ on a nonempty convex
subset $S$ of $\Delta(X)$ is said to admit a {\it Bernoulli utility} if there
exists a Borel measurable function $u: X \to \mathbb{R}$ such that
\begin{equation}\label{bernoulli}
p\succsim q\qquad\mbox{iff}\qquad
\int_X u\, {\rm d}p \geq \int_X u\, {\rm d}q
\end{equation}
for every $p,q \in S$. In this case we say that $u$ is a {\it Bernoulli utility function}
for $\succsim$. The following theorem is one of the most famous results in decision theory:
\bigskip

\noindent {\bf The expected utility theorem.}
{\it Let $X$ be a separable metric space and $\succsim$ a total preorder on
$\Delta(X)$. Then $\succsim$ admits a continuous and bounded Bernoulli
utility if and only if it is continuous and affine.}
\bigskip

This theorem was first proved by von Neumann and Morgenstern \cite{vnm}
in the case where $X$ is finite. The version given above is due to Grandmont
\cite{grandmont}. The theorem says that so long as one's preferences over
risky prospects (modeled here as the elements of $\Delta(X)$) satisfy the
rationality traits of transitivity and affinity, are decisive (captured here by
the totalness property), and are continuous (a natural regularity
property), then they can be thought of as evaluating any given risky prospect
$p$ by means of the expected utility of $p$, with respect to some
(continuous and bounded) utility function $u: X \to \mathbb{R}$ (over
riskless alternatives). The term ``Bernoulli utility function'' comes from
decision theory due to Bernoulli's famous resolution of the St.\ Petersburg
paradox. The map $p \mapsto \int u\, {\rm d}p$ is often referred to as a
{\it von Neumann-Morgenstern utility function}.

\subsection{The expected multi-utility theorem}
The one assumption of the expected utility theorem that does not seem to
be directly related to rationality is the so-called {\it completeness axiom},
which states that the preorder modelling the underlying preference relation
must be total. Since the seminal contribution of Aumann \cite{aumann},
it has been argued by numerous authors that rational decision makers
may well remain indecisive about outcomes that they find difficult to
compare, or have insufficient information about. This led decision
theorists to search for a functional representation that is suitable for
continuous affine preorders over risky prospects that need not be total.
In fact, \cite{vnm} conjectured, albeit informally, that one could represent
such a preorder by means of the expectation of a vector-valued Bernoulli
utility function. For any metric space $X$, let us say that a preorder $\succsim$
on a nonempty convex subset $S$ of $\Delta(X)$ admits a {\it Bernoulli multi-utility}
if there exists a family $\mathcal{U}$ of Borel measurable real-valued functions on
$X$ such that, for any $p, q \in S$,
\begin{equation}\label{multiorder}
p\succsim q\qquad\mbox{iff}\qquad
\int_X u\, {\rm d}p \geq \int_X u\, {\rm d}q\quad\mbox{for every }u \in \mathcal{U}.
\end{equation}
In this case, we refer to $\mathcal{U}$ as a {\it Bernoulli multi-utility} for
$\succsim$. We may qualify $\mathcal{U}$ with adjectives such as ``continuous''
or ``bounded'' if each member of $\mathcal{U}$ has the same property.

In the case where the riskless prize space $X$ is compact, Dubra, Maccherroni, and
Ok \cite{dmo} showed that one can relax the completeness axiom in the expected
utility theorem simply by replacing Bernoulli utilities with Bernoulli multi-utilities.
This is known as:
\bigskip

\noindent {\bf The expected multi-utility theorem.}
{\it Let $X$ be a compact metric space and $\succsim$ a preorder on
$\Delta(X)$. Then $\succsim$ admits a continuous Bernoulli
multi-utility if and only if it is continuous and affine.\footnote{\cite{evren08}
showed that the compactness hypothesis cannot be relaxed to separability
here so long as one insists on the continuity of Bernoulli multi-utilities. However,
when $X$ is $\sigma$-compact and $\succsim$ is an affine preorder on the set
of all Borel probability measures on $X$ with compact support, $\succsim$
does admit a continuous Bernoulli multi-utility. Given its foundational value,
it is also of interest to see how the expected utility theorem would change
if we took away from it any subset of the hypotheses of totalness, transitivity,
and continuity. This query was recently settled in \cite{hor}.}}
\bigskip

One limitation of these results is that they operate only with bounded Bernoulli
utilities. In fact, the expected multi-utility theorem only applies to a compact
domain; one cannot even set $X = \mathbb{R}$ in its statement. By contrast, in
practice (say, in finance), it is commonplace to work with unbounded utility
functions over prize spaces such as $\mathbb{R}$. In addition, in applications
one often imposes conditions on Bernoulli utility functions which are not warranted by
either of the above theorems. In particular, it is fairly common to work
with differentiable, or at least almost everywhere differentiable, Bernoulli utility
functions on $\mathbb{R}$. This further severs the connection between the
behavioral foundations of the expected (multi-)utility model from its applications.

As we discussed in the introduction, our primary objective here is to identify
precisely which type of affine preorders admit a Bernoulli (multi-)utility that is
Lipschitz continuous on a separable metric space $X$ (but restricting the domain
of preferences to $\Delta_1(X)$). This will largely eliminate the above-mentioned
limitations, because such functions need not be bounded and possess desirable
differentiability properties.

\subsection{Advantages of Lipschitz Bernoulli utilities}
It is well-known that Lipschitz functions behave well from the differentiability
standpoint. In particular, when $X$ is a positive measure subset of a
finite-dimensional Euclidean space (as in most applications of expected
utility theory), any Lipschitz function on $X$ is
differentiable almost everywhere (Rademacher's theorem). There are
infinite-dimensional versions of this result as well.  For instance, Phelps
\cite{phelps} showed that every Lipschitz function on a separable Banach
space is Gateaux differentiable everywhere but on a Gaussian null set. And
from the purely topological perspective, when $X$ is an Asplund space ---
for instance, when $X$ is any Banach space with a separable dual, or any
reflexive Banach space --- a famous theorem of Preiss \cite{preiss} shows that
any Lipschitz function on $X$ is Fr\'{e}chet differentiable on a dense subspace of
$X$. Thus, behavioral characterizations of preorders that admit Lipschitz
Bernoulli utilities, such as the one obtained in the next section, build a bridge
between behavioral properties (that is, conditions that are expressible by
means of the preorder alone) and the seemingly technical property of
differentiability.

Another advantage of Lipschitz functions concerns their optimization. If only
to motivate interest in Lipschitz Bernoulli utility functions, we illustrate this
by means of a simple example from finance. Consider a market for financial
securities with $n$ many different stocks. Let us model the returns from these
stocks by an $\mathbb{R}^n$-valued random variable $a$ on some probability
space, whose joint distribution is given by a Borel probability measure $\mu$
on $\mathbb{R}^n$ with bounded support. A {\it portfolio} in this setting is
any $n$-vector $\alpha$ whose expected utility for an investor with a continuous
Bernoulli utility $u \in C(\mathbb{R})$ is $F(\alpha)$, where $F: \mathbb{R}^n
\to \mathbb{R}$ is defined as $F(\alpha) := \int u(\alpha \cdot a)\,
\mu({\rm d}a)$. Suppose the current wealth of the investor is $W > 0$. Then the
optimal portfolio choice problem for this individual is to maximize $F(\alpha)$
over all $\alpha \in \mathbb{R}^n$ with $\pi\cdot \alpha \leq W$, where $\pi$
is the given vector of unit stock prices.

Differential calculus is of little help in studying this simple portfolio optimization
problem. But if $u$ is Lipschitz, we can say quite a bit about it, and may even be
able to solve it precisely, using nonsmooth analysis. Indeed, if $u$ is Lipschitz,
then so is $F$ because
\begin{eqnarray*}
F(\alpha) - F(\beta) &=& \int_{\mathbb{R}^n} \left(u(\alpha\cdot a) -
u(\beta\cdot a)\right)\, \mu({\rm d}a)\\
&\leq& L \int_{\mathbb{R}^n} |(\alpha - \beta)\cdot a|\, \mu({\rm d}a)\\
&\leq& K\|\alpha - \beta\|
\end{eqnarray*}
for every $\alpha, \beta \in \mathbb{R}^n$, where $L$ is the Lipschitz number of
$u$ and $K := L\int \|\cdot\|\, {\rm d}\mu$. (Here $\|\cdot\|$
stands for the Euclidean norm.) It follows that if $\alpha^*$ is a solution to our
maximization problem, we have ${\bf 0} \in \partial F(\alpha^*)$, where $\partial$ is
(Clarke's) generalized gradient operator, which is well-defined since $u$ is
Lipschitz. Next, for any open subset $O$ of $\mathbb{R}^n$,
$\partial F(\alpha^*)$ is the convex hull of the set of all vectors of the form
$\nabla F(\alpha_m)$ where $(\alpha_m)$ is a sequence in $O$ such that $F$
is differentiable at each $\alpha_m$ and $\alpha_m \to \alpha^*$. Combining
this fact with Carath\'{e}odory's theorem, we may conclude that for every
sufficiently small $\epsilon > 0$ there exist $n+1$ many sequences
$(\alpha_{m,1}), \ldots, (\alpha_{m,n+1})$ and a vector $(\lambda_1, \ldots,
\lambda_{n+1})$ in the $n$-dimensional unit simplex such that (i) $F$ is
differentiable at $\alpha_{m,i}$ for each $m \geq 1$ and $i = 1, \ldots, n+1$;
(ii) $\alpha_{m,i} \to \alpha^*$ for each $i = 1, \ldots, n+1$; and (iii)
\begin{equation*}
\lim_m \sum_{i=1}^{n+1} \lambda_i \nabla F(\alpha_{m,i}) = {\bf 0}.
\end{equation*}
Thus, thanks to the Lipschitz property of $u$ alone, differential calculus enters
into the analysis of the optimal portfolio problem from the back door.

Due to such advantages of Lipschitz functions, one may choose to work within
environments in which individuals are modeled as expected utility maximizers
with Lipschitz Bernoulli utilities. (Indeed, in applications one often restricts attention
either to such functions or to functions with additional structure.) It is not a priori
clear, however, what sort of preference relations would actually allow for an expected
utility representation with a Lipschitz Bernoulli (multi-)utility. We give our answer to
this question in the next section.

\section{Lipschitz Bernoulli utilities}

\subsection{A Lipschitz property for preorders on $\Delta_1(X)$}

We will study the following property.

\begin{defi}\label{lipdef}
Let $X$ be a separable metric space. We say that a preorder
$\succsim$ on $\Delta_1(X)$ is {\it Lipschitz} if for every $p, q \in \Delta_1(X)$ such
that $q \succsim p$ fails, there is a $K = K(p,q) > 0$ such that, for any $p', q' \in \Delta_1(X)$,
\begin{equation}\label{lipcon}
q \underset{\lambda}{\oplus} q' \succsim p \underset{\lambda}{\oplus} p'\quad
\mbox{fails whenever }0 \leq \lambda < \frac{K}{K + W_1(p',q')}.
\end{equation}
\end{defi}

This condition quantitatively expresses the idea that if $q \succsim p$ fails, then the same
comparison between the compound lotteries\footnote{Here, as in the justification of affinity,
we regard $r \oplus_\lambda r'$ as a compound lottery in which we first chose $r$
with probability $1-\lambda$ and $r'$ with probability $\lambda$, and then use the chosen
measure to determine a reward in $X$.}
$q \oplus_\lambda q'$ and $p \oplus_\lambda p'$ should also fail, provided $\lambda$
is sufficiently small. Continuity of $\succsim$ would only ensure that sliding $p$ and $q$ towards
two other points in the space $\Delta_1(X)$ by a small amount would not change
their order relation; the substance of
(\ref{lipcon}) lies in the specific form of the upper bound $\frac{K}{K + W_1(p',q')}$ on $\lambda$.
(Although we do not include any continuity assumption in Definition \ref{lipdef}, it will
follow from Theorem \ref{mainthm} below that the stated Lipschitz condition implies
$W_1$-continuity.)

This upper bound expresses a sort of uniformity of preferences in relation to the given metric
on $X$. It can be informally motivated as follows.

Suppose $q \succsim p$ fails. We can characterize how badly it fails by considering
for which $p'$ and $q'$ we do have $q \oplus_{1/2}q' \succsim p \oplus_{1/2} p'$. We might
say that these are the $p'$ and $q'$ for which $q'$ is preferred to $p'$ strongly enough that
even a 50\% chance of getting $q'$ (and 50\% chance of getting $q$) is unambiguously better
than a 50\% chance of getting $p'$ (and a 50\% chance of getting $p$). We therefore define
\begin{equation*}
K = \inf\left\{a > 0:
q \underset{1/2}{\oplus} q' \succsim p \underset{1/2}{\oplus} p'
\mbox{ for some $p'$, $q'$ with }W_1(p', q') < a\right\}.
\end{equation*}
This is the critical distance such that if $p'$ and $q'$ are any closer, then even if $q'$ is
preferred to $p'$, getting $q$ or $q'$ with even odds still will not be unambiguously preferable to getting $p$
or $p'$ with even odds. If we interpret $W_1(p',q')$ as measuring how similar $p'$ and $q'$
are\footnote{We might first interpret $d(x,y)$ as measuring the similarity of two prizes
$x,y \in X$; then the optimal transportation interpretation of the Wasserstein 1-metric makes
$W_1$ a natural measure of similarity in $\Delta_1(X)$.}, then it is reasonable to assume that
$K > 0$, on the grounds that, given $q \succsim p$ fails, $q \oplus_{1/2} q' \succsim p \oplus_{1/2} p'$
should also fail whenever $p'$
and $q'$ are sufficiently similar. Thus, for all $p'$ and $q'$ with $W_1(p',q')$ less than the critical
value $K$, the lottery which gives even odds of $q$ or $q'$ is not
preferred to the lottery which gives even odds of $p$ or $p'$. The key assumption now is
that this conclusion scales; for example, if $W_1(p', q') < 2K$ then $q \oplus_{1/3} q'$
is not preferred to $p \oplus_{1/3} p'$ --- the reasoning being that since $W_1(p',q')$ is now
less than twice that critical value, 2:1 odds are insufficient compensation. The general principle
is that if $W_1(p',q') < aK$ for some $a > 0$, then $a:1$ odds are insufficient compensation;
$q \oplus_\lambda q' \succsim p \oplus_\lambda p'$ must fail for
any $\lambda < \frac{1}{1 + a}$, since $\lambda_0 = \frac{1}{1 + a}$ is the value at which
$\frac{1 - \lambda_0}{\lambda_0} = \frac{a}{1}$. This shows that
$\lambda < \frac{K}{K + W_1(p',q')}$ --- obtained by putting $a = W_1(p', q')/K$ in $\lambda <
\frac{1}{1 + a}$ ---
will ensure that $q \oplus_\lambda q' \succsim p \oplus_\lambda p'$ fails, which is precisely what
(\ref{lipcon}) says.

Let us now verify that preorders arising from Lipschitz Bernoulli multi-utilities
are Lipschitz in the sense of Definition \ref{lipdef}.

\begin{prop}\label{liputilities}
Let $X$ be a separable metric space and let $\mathcal{U}$ be a family of Lipschitz
functions on $X$. Then the preorder $\succsim$ on $\Delta_1(X)$ defined by
\begin{equation*}
p\succsim q\qquad\mbox{iff}\qquad
\int_X u\, {\rm d}p \geq \int_X u\, {\rm d}q\quad\mbox{for every }u \in \mathcal{U}
\end{equation*}
is affine and Lipschitz.
\end{prop}

\begin{proof}
Affinity is straightforward. To verify the Lipschitz condition, take any $p, q \in \Delta_1(X)$
for which $q \succsim p$ fails. Then there exists $u \in \mathcal{U}$ such that $U(p) > U(q)$,
where $U: \Delta_1(X) \to \mathbb{R}$ is defined by $U(p) := \int u\, {\rm d}p$.
Since $u$ can be scaled without affecting the preorder, we may assume without loss
of generality that its Lipschitz number is 1. It follows from the Kantorovich-Rubinstein
theorem (recall (\ref{krthm})) that
\begin{equation}\label{ineq}
U(q') - U(p') = \int_X u\, {\rm d}(q' - p') \leq W_1(p',q')
\end{equation}
for any $p', q' \in \Delta_1(X)$.

Set $K := U(p) - U(q)$, let $p',q' \in \Delta_1(X)$, and choose
$0 \leq \lambda < \frac{K}{K + W_1(p',q')}$; we wish to show that
$q \oplus_\lambda q' \succsim p \oplus_\lambda p'$ fails. But
\begin{eqnarray*}
U(q \underset{\lambda}{\oplus} q')
&=& (1 - \lambda)U(q) + \lambda U(q')\cr
&\leq& (1 - \lambda)(U(p) - K) + \lambda(U(p') + W_1(p',q'))\cr
&=& (1 - \lambda)U(p) + \lambda U(p') - (1 - \lambda)K + \lambda W_1(p',q')\cr
&=& U(p \underset{\lambda}{\oplus} p') + \lambda(K + W_1(p',q')) - K\cr
&<& U(p \underset{\lambda}{\oplus} p') + K - K\cr
&=& U(p \underset{\lambda}{\oplus} p'),
\end{eqnarray*}
using (\ref{ineq}) in the first inequality and the choice of $\lambda$ in the second.
Thus $q \oplus_\lambda q' \succsim p \oplus_\lambda p'$ indeed fails, as desired.
\end{proof}

The Lipschitz property defined above seems new. A vaguely similar condition has been introduced
in \cite{dlrs} and studied further in \cite{sarver, es1, es2}, which applies to total preorders defined
on the collection of all nonempty subsets of $\Delta(X)$ for some finite set $X$. More topically,
there is a definition of a Lipschitz property for preorders defined on any metric space due to Levin
\cite{levin} which is immediately applicable to the present setting. We will explain the connection
between our definition and that of \cite{levin} for affine preorders in Section 4.7.

\subsection{Stochastic orders are Lipschitz}
Proposition \ref{liputilities} allows one to identify many interesting affine Lipschitz preorders.
In this section, we verify that the usual (multivariate) stochastic orders are Lipschitz. Such partial
orders are quite important, as they are widely used in a plethora of fields, ranging from the theories
of risk measurement and reliability to quantum mechanics. (See \cite{ss} for an extensive treatment
of stochastic orders of various types.)

The {\it univariate stochastic order}  $\succcurlyeq_{\rm st}$ on $\Delta(\mathbb{R})$
is defined as
\begin{equation*}
p \succcurlyeq_{\rm st} q\qquad\mbox{iff}\qquad
p((-\infty, a]) \leq q((-\infty, a])\hbox{ for every }a \in \mathbb{R}.
\end{equation*}
(Economists refer to this partial order as {\it first-order stochastic dominance}
and use it to determine if one monetary lottery is better than another in an
unambiguous sense.) For any $a \in \mathbb{R}$ and $n \in \mathbb{N}$ let
$u_{a,n}$ be the continuous function which is constantly $0$ on $(-\infty, a]$, constantly $1$
on $[a + \frac{1}{n}, \infty)$, and linear on $[a, a + \frac{1}{n}]$.
Since $p((-\infty, a]) = 1 - \lim_{n\to \infty} \int u_{n,a}\, {\rm d}p$, it follows that
$\mathcal{U} := \{u_{a,n}: a \in \mathbb{R}$ and $n \in \mathbb{N}\}$ is a collection
of Lipschitz functions on $\mathbb{R}$ such that $p\succcurlyeq_{\rm st} q$ iff $\int u\, {\rm d}p
\geq \int u\, {\rm d}q$ for each $u \in \mathcal{U}$. Thus $\succcurlyeq_{\rm st}$ is an affine
Lipschitz partial order on $\Delta_1(\mathbb{R})$.

This finding extends well beyond the univariate case. Suppose $X = (X, \geq, +, d)$ is any (partially)
ordered abelian group equipped with a separable translation-invariant metric.  A {\it lower subset}
of $X$ is a subset $S$ with the property that $x \in S$ and $x \geq y$ imply $y \in S$. We define the
{\it stochastic order $\succcurlyeq_{\rm st}$} on $\Delta(X)$ induced by $\geq$ by
\begin{equation*}
p \succcurlyeq_{\rm st} q\qquad\mbox{iff}\qquad
p(S) \leq q(S)\quad\mbox{for every closed lower subset $S$ of $X$}.
\end{equation*}
The following well-known characterization of $\succcurlyeq_{\rm st}$ was obtained by
Kamae, Krengel, and O'Brien \cite{kko}: $p \succcurlyeq_{\rm st} q$ iff $\int u\, {\rm d}p
\geq \int u\, {\rm d}q$ for every bounded and $\geq$-increasing Borel measurable
$u: X \to \mathbb{R}$. But in fact ``Borel measurable'' can be replaced by ``Lipschitz'' in
this characterization. This is a consequence of the following lemma.

\begin{lemma}
Let $X = (X, \geq, +, d)$ be an ordered abelian group equipped with a separable
translation-invariant metric and let $S$ be a lower subset of $X$. Then
$d(\cdot, S)$ is a $\geq$-increasing function.
\end{lemma}

\begin{proof}
Take any $x, y \in X$ with $x \geq y$. Then for any $z \in S$, adding $z - x$ to both sides of
$x \geq y$ yields $z \geq z - x + y$, so that $z - x + y \in S$. Thus $d(x,z) = d(y, z - x + y) \geq d(y,S)$,
and taking the infimum over $z$ yields $d(x, S) \geq d(y,S)$. This shows that $d(\cdot, S)$ is
$\geq$-increasing.
\end{proof}

\begin{prop}
Let $X = (X, \geq, +, d)$ be an ordered abelian group equipped with a separable translation-invariant
metric. Then the stochastic order on $\Delta_1(X)$ induced by $\geq$ is Lipschitz and affine.
\end{prop}

\begin{proof}
For any closed lower subset $S$ of $X$ and any $n > 0$, let $u_{S, n}$ be the function
$u_{S, n}(\cdot) = {\rm min}(nd(\cdot, S), 1)$.
Each of these functions is clearly Lipschitz, and is $\geq$-increasing by the previous lemma.
Let $\mathcal{U} = \{u_{S, n}: S$ is a closed lower subset of $X$ and $n > 0\}$. We will show
that $p \succcurlyeq_{\rm st} q$ iff $\int u\, {\rm d}p \geq \int u\, {\rm d}q$ for every
$u \in \mathcal{U}$; the desired conclusion will then follow from Proposition \ref{liputilities}.

Since every $u \in \mathcal{U}$ is bounded, $\geq$-increasing, and Borel, the forward
implication follows from the result from \cite{kko} mentioned above. For the
reverse implication, suppose $\int u\, {\rm d}p \geq \int u\, {\rm d}q$
for every $u \in \mathcal{U}$. Fix a closed lower subset $S$ of $X$. Then as $n \to \infty$
the functions $u_{S, n}$ converge pointwise to $1_{X\setminus S}$, so it follows from the
dominated convergence theorem that
\begin{equation*}
p(X \setminus S) = \int_X 1_{X \setminus S}\, {\rm d}p \geq
\int_X 1_{X \setminus S}\, {\rm d}q = q(X \setminus S),
\end{equation*}
and therefore $p(S) \leq q(S)$. Since $S$ was an arbitrary closed lower subset of $X$, we
conclude that $p \succcurlyeq_{\rm st} q$.
\end{proof}

Thus, virtually any stochastic order one encounters in practice is a Lipschitz affine
partial order.

\subsection{The expected Lipschitz multi-utility theorem}
The goal of this section is to prove a converse to Proposition \ref{liputilities} which states
that any Lipschitz affine preorder $\succsim$ on $\Delta_1(X)$ admits a Lipschitz
Bernoulli multi-utility. This is the main result of the paper. We will also show that in the
special case where $\succsim$ is total, $\mathcal{U}$ can be taken to consist of a single
Lipschitz Bernoulli utility function $u$ for $\succsim$ (i.e., (\ref{bernoulli}) holds).

\begin{lemma}\label{mainlemma}
Let $X$ be a separable metric space and let $\succsim$ be a Lipschitz affine preorder
on $\Delta_1(X)$. Then $C := \{\alpha(p - q): \alpha \geq 0$ and $p\succsim q\}$ is a
closed and convex subset of ${\rm KR}(X)$, and for any $p, q \in \Delta_1(X)$ we have
\begin{equation}\label{corder}
p \succsim q\qquad\mbox{iff}\qquad p - q \in C.
\end{equation}
\end{lemma}

\begin{proof}
It follows from the second part of Lemma \ref{basiclemma} that $\{p - q: p \succsim q\}$
is convex, and therefore $C$, the cone this set generates, is also convex. Next, we verify
(\ref{corder}). The forward implication is immediate; for the reverse implication, suppose
$p - q \in C$ and write $p - q = \alpha(p' - q')$ with $\alpha \geq 0$ and $p' \succsim q'$. Setting
$\lambda := \frac{\alpha}{\alpha + 1}$, we then have $(1 - \lambda)(p - q) = \lambda(p' - q')$,
and rearranging this yields $p \oplus_\lambda q' = q \oplus_\lambda p'$. Now by the first part of Lemma
\ref{basiclemma}, $p' \succsim q'$ implies $p \oplus_\lambda p' \succsim p \oplus_\lambda q'
= q \oplus_\lambda p'$, so that $p \succsim q$, as desired.

It remains to show that $C$ is closed. Fix $\sigma \in {\rm KR}(X)\setminus C$ and write
$\sigma = \alpha(q - p)$ for some $\alpha > 0$ and $p,q \in \Delta_1(X)$ (Proposition \ref{krprop}).
Note that $q \succsim p$
must fail, since $\sigma \not\in C$. Let $K$ be as in Definition \ref{lipdef} for this $p$ and $q$, so that
$q \oplus_\lambda q' \succsim p \oplus_\lambda p'$ fails whenever $0 \leq \lambda <
\frac{K}{K + W_1(p', q')}$, for any $p', q' \in \Delta_1(X)$. We will show that the open
ball of radius $\alpha K$ about $\sigma$ is disjoint from $C$, and since $\sigma$ is an arbitrary
element in the complement of $C$, this will establish that $C$ is closed.

We claim that for every $\mu \in {\rm KR}(X)$ we have $\sigma \oplus_\theta \mu \not\in C$
whenever $0 \leq \theta < \frac{\alpha K}{\alpha K + \|\mu\|_{\rm KR}}$. If $\mu = 0$, this follows
from the fact that $C$ is a cone; otherwise, write
$\mu = \beta(q' - p')$ with $\beta > 0$ and $p', q' \in \Delta_1(X)$ (Proposition \ref{krprop} again).
Then for any $0 \leq \lambda < \lambda_0 := \frac{K}{K + W_1(p', q')}$,
the choice of $K$ ensures that $q \oplus_\lambda q' \succsim p \oplus_\lambda p'$
fails, and therefore by (\ref{corder})
\begin{equation*}
\left(q \underset{\lambda}{\oplus} q'\right) - \left(p \underset{\lambda}{\oplus} p'\right)
= (1 - \lambda)(q - p) + \lambda(q' - p') \not\in C,
\end{equation*}
or equivalently
\begin{equation*}
\frac{1 - \lambda}{\alpha}\sigma + \frac{\lambda}{\beta}\mu \not\in C.
\end{equation*}
Since $C$ is a cone, this will still be true if the left side is multiplied by any positive scalar; in particular,
\begin{equation*}
\frac{(1-\lambda)/\alpha}{(1 - \lambda)/\alpha + \lambda/\beta}\sigma
+ \frac{\lambda/\beta}{(1 - \lambda)/\alpha + \lambda/\beta}\mu \not\in C
\end{equation*}
for $0 \leq \lambda < \lambda_0$. That is, for any $\lambda$ in this range we have
$\sigma \oplus_\theta \mu \not\in C$ where
$\theta = \theta(\lambda) := \frac{\lambda/\beta}{(1-\lambda)/\alpha + \lambda/\beta}$.
Since $\theta(\lambda) = \frac{\alpha}{\beta(1/\lambda - 1) + \alpha}$ is a continuous increasing
function of $\lambda$, this means that $\sigma \oplus_\theta \mu \not\in C$ whenever
\begin{eqnarray*}
0\quad\leq\quad \theta\quad <\quad \theta(\lambda_0)
&=& \frac{\lambda_0/\beta}{(1-\lambda_0)/\alpha + \lambda_0/\beta}\cr
&=& \frac{1}{1 + (1 - \lambda_0)\beta/(\lambda_0 \alpha)}\cr
&=& \frac{1}{1 + (1/\lambda_0 - 1)(\beta/\alpha)}\cr
&=& \frac{1}{1 + (W_1(p',q')/K)(\beta/\alpha)}\cr
&=& \frac{\alpha K}{\alpha K + \|\mu\|_{\rm KR}}.
\end{eqnarray*}
This proves the claim.

To complete the proof, choose any $\mu' \in {\rm KR}(X)$ with $\|\mu'\|_{\rm KR} = 1$.
We will show that $\sigma + \gamma\mu' \not\in C$ for $0 \leq \gamma < \alpha K$,
yielding that the open ball of radius $\alpha K$ about $\sigma$ does not
intersect $C$. To do this, for each $N \in \mathbb{N}$ define
$\mu_N := \sigma + N\mu'$. Then according to the claim, we have
\begin{equation*}
(1 - \theta)\sigma + \theta\mu_N = \sigma + N\theta\mu' \not\in C
\end{equation*}
whenever $0 \leq \theta < \theta_N := \frac{\alpha K}{\alpha K + \|\mu_N\|_{\rm KR}}$. That is,
$\sigma + \gamma\mu' \not\in C$ provided
\begin{equation*}
0 \leq \gamma < N\theta_N = \frac{\alpha NK}{\alpha K + \|\mu_N\|_{\rm KR}}
=\frac{\alpha K}{\alpha K/N + \|\mu_N\|_{\rm KR}/N}.
\end{equation*}
As $N \to \infty$ we have $\alpha K/N \to 0$ and $\|\mu_N\|_{\rm KR}/N
= \|\sigma + N\mu'\|_{\rm KR}/N \to \|\mu'\|_{\rm KR} = 1$. So taking $N \to \infty$
yields $\sigma + \gamma\mu' \not\in C$ for $0 \leq \gamma < \alpha K$, as desired.
\end{proof}

\begin{theo}\label{mainthm}
Let $X$ be a separable metric space and let $\succsim$ be a preorder on $\Delta_1(X)$.
Then $\succsim$ is Lipschitz and affine if and only if there exists a family $\mathcal{U}$
of Lipschitz functions on $X$ such that
\begin{equation}\label{multiorder2}
p\succsim q\qquad\mbox{iff}\qquad
\int_X u\, {\rm d}p \geq \int_X u\, {\rm d}q\quad\mbox{for every }u \in \mathcal{U}.
\end{equation}
\end{theo}

\begin{proof}
The reverse implication in this characterization is precisely Proposition \ref{liputilities}. To
establish the forward implication,  assume $\succsim$ is Lipschitz and affine. It will suffice
to find, for every $p,q \in \Delta_1(X)$ such that $q \succsim p$ fails, a Lipschitz function
$u = u_{p,q}$ on $X$ which satisfies both
\begin{equation}\label{need1}
\int_X u\, {\rm d}p > \int_X u\, {\rm d}q
\end{equation}
and
\begin{equation}\label{need2}
\int_X u\, {\rm d}p' \geq \int_X u\, {\rm d}q'
\end{equation}
for every $p', q' \in \Delta_1(X)$ with $p' \succsim q'$. We can then take $\mathcal{U} = \{u_{p,q}:
q \succsim p$ fails$\}$ in (\ref{multiorder2}).

As in Lemma \ref{mainlemma}, define $C := \{\alpha(p - q): \alpha \geq 0$ and $p\succsim q\} \subseteq {\rm KR}(X)$.
Fix $p,q \in \Delta_1(X)$ such that $q \succsim p$ fails. Then $q - p \not\in C$ by the last part of
that lemma, and since $C$ is a closed convex cone, the separating hyperplane theorem yields a bounded
linear functional $U$ on ${\rm KR}(X)$ such that $U(\mu) \geq 0 > U(q - p)$ for every $\mu \in C$.
By Theorem \ref{duality}, there is a Lipschitz function $u$ on $X$ with
$U(\mu) = \int u\, {\rm d}\mu$ for any $\mu \in {\rm KR}(X)$. Putting $\mu = q - p$, this yields
(\ref{need1}), and putting $\mu = p' - q' \in C$ for any $p', q' \in \Delta_1(X)$ with $p' \succsim q'$, it yields
(\ref{need2}). This completes the proof.
\end{proof}

If $\succsim$ is total, then it is not too hard to see that the family $\mathcal{U}$ in
Theorem \ref{mainthm} can be taken to consist of a single utility function $u$. We thus obtain the
following Lipschitz version of the classical von Neumann-Morgenstern expected utility theorem.

\begin{coro}\label{maincor}
Let $X$ be a separable metric space and let $\succsim$ be a total preorder on $\Delta_1(X)$. Then
$\succsim$ is Lipschitz and affine if and only if there is a Lipschitz function $u$ on $X$ such that
\begin{equation}\label{bernoulli2}
p\succsim q\qquad\mbox{iff}\qquad
\int_X u\, {\rm d}p \geq \int_X u\, {\rm d}q.
\end{equation}
\end{coro}

\begin{proof}
The reverse implication is a special case of Proposition \ref{liputilities}. For the forward implication,
assume $\succsim$ is Lipschitz and affine. We know from Theorem \ref{mainthm} that there is a
Bernoulli multi-utility $\mathcal{U} \subseteq {\rm Lip}_0(X)$ for $\succsim$. We
must show that there is a single Bernoulli utility function $u \in {\rm Lip}_0(X)$.

If $\mathcal{U} = \emptyset$ or $\{0\}$ then $p \sim q$ for every $p, q \in \Delta_1(X)$ and we
can take $u = 0_X$ in (\ref{bernoulli2}). Otherwise, fix any nonzero $u \in \mathcal{U}$ and
define $U: \Delta_1(X) \to \mathbb{R}$
by $U(p) := \int u\, {\rm d}p$. We know that $p \succsim q$ implies $U(p) \geq U(q)$, and we
must prove the reverse implication. That is --- since $\succsim$ is total --- we must show that
$U(p) = U(q)$ is impossible if $p \succ q$. It will then follow
that condition (\ref{bernoulli2}) holds with this $u$.

Suppose for the sake of contradiction that $p \succ q$ but $U(p) = U(q)$. Since $u$ is not constant,
there exists $r \in \Delta_1(X)$ with $U(r) \neq U(p) = U(q)$; without essential loss of generality
suppose $U(r) < U(p) = U(q)$. We must then have $p \succ q \succ r$.

Since $p \succ q$ and $\succsim$ is Lipschitz and total, we have
$p \oplus_\lambda r \succ q \oplus_\lambda q = q$ for sufficiently small $\lambda > 0$. However,
\begin{equation*}
U\left(p \underset{\lambda}{\oplus} r\right) = (1 - \lambda)U(p) + \lambda U(r) < U(p) = U(q),
\end{equation*}
implying that $p \oplus_\lambda r \prec q$, a contradiction. This completes the proof.
\end{proof}

Corollary \ref{maincor} can also be proven directly, without invoking Theorem \ref{mainthm}.
The case where $\succsim$ is total is special because any nonconstant function $u$ which
satisfies (\ref{bernoulli2}) is, in fact,
completely determined by its values on any two inequivalent (i.e., indifferent)
points of $X$. So the Hahn-Banach aspect
of the proof of Theorem \ref{mainthm} is not really relevant --- in this case the cone $C$
becomes a half-space and no separation argument is needed. Instead, we could simply fix the values
of $u$ on two inequivalent points in $X$, determine what its value must be on any other point, and
then verify that the resulting function must be Lipschitz. However, it is easier to deduce Corollary
\ref{maincor} from Theorem \ref{mainthm}, so we omit this longer but more elementary proof.

\subsection{Uniqueness of Lipschitz multi-utilities}
Let $X$ be a separable metric space and let $\mathcal{U}$ be a family of Lipschitz functions on $X$.
According to Proposition \ref{liputilities}, $\mathcal{U}$ induces a Lipschitz affine preorder
$\succsim$ on $\Delta_1(X)$ via (\ref{multiorder2}).

We can modify $\mathcal{U}$ in various ways without affecting the preorder it induces. For
example, if $u$ and $v$ are related by a positive affine transformation then
\begin{equation*}
\int_X u\, {\rm d}p \geq \int_X u\, {\rm d}q\qquad\mbox{ if and only if }\qquad
\int_X v\, {\rm d}p \geq \int_X v\, {\rm d}q,
\end{equation*}
for any $p, q \in \Delta_1(X)$. Thus every function in $\mathcal{U}$ can be replaced by a positive
affine transformation of itself without affecting the induced preorder. Also, constant functions
in $\mathcal{U}$ have no effect on the induced preorder, so they can be included or removed
without consequence.

Fixing a base point $e \in X$ (as in Section 2.2) and denoting the Lipschitz number of $u$ by $L(u)$, this observation
yields that the set $\mathcal{U}_1 = \left\{\frac{u - u(e)}{L(u)}: u \in \mathcal{U}\mbox{ is nonconstant}\right\}$
induces the same preorder $\succsim$ as $\mathcal{U}$. Thus, we can always replace $\mathcal{U}$
by  a subset of the unit sphere of ${\rm Lip}_0(X)$ without affecting $\succsim$. At the other extreme,
once $\succsim$ has been identified,
we can take $\mathcal{U}_2$ to be the set of {\it all} $u \in {\rm Lip}_0(X)$ which satisfy
$\int u\, {\rm d}p \geq \int u\, {\rm d}q$ for every $p, q \in \Delta_1(X)$ such that $p \succsim q$.
This family will also induce the same preorder, and it is clearly the unique maximal subset of ${\rm Lip}_0(X)$
which does so. This $\mathcal{U}_2$ is just the cone which is dual to the cone $C$ used in the proof of
Theorem \ref{mainthm}, relative to the duality ${\rm KR}(X)^* \cong {\rm Lip}_0(X)$ of Theorem
\ref{duality}.

Any family $\mathcal{V} \subseteq {\rm Lip}_0(X)$ satisfying $\mathcal{U}_1 \subseteq
\mathcal{V} \subseteq \mathcal{U}_2$ will induce the same preorder as $\mathcal{U}$. Together with
the easy observation that $\mathcal{U}_2$ is weak* closed and convex, this shows that in particular
the weak* closed convex hull of $\mathcal{U}_1$) induces the same preorder as $\mathcal{U}$. Thus,
not only can $\mathcal{U}$ always be replaced by a subset of the unit sphere of ${\rm Lip}_0(X)$, it can
also be replaced by a weak* closed convex subset of the closed unit ball of ${\rm Lip}_0(X)$. We have
shown:

\begin{coro}\label{convexmulti}
Let $X$ be a pointed separable metric space and let $\succsim$ be a preorder on $\Delta_1(X)$.
Then $\succsim$ is Lipschitz and affine if and only if some weak* closed convex subset
of the closed unit ball of ${\rm Lip}_0(X)$ a Bernoulli multi-utility for $\succsim$.
\end{coro}

We can also formulate a version of the uniqueness theorem of \cite{dmo} in the present setting.
However, here we have to use a slightly weaker topology than the standard weak* topology on
${\rm Lip}_0(X)$, namely the weakest topology which makes integration against every measure
in ${\rm KR}(X)$ continuous. We will call this the {\it {\rm KR}-weak* topology}. It agrees with
the standard weak* topology on bounded subsets of ${\rm Lip}_0(X)$ (where both are just the
topology of pointwise convergence), but unless ${\rm KR}(X)$ is complete it will be weaker than
the standard weak* topology on unbounded sets.

\begin{theo}
Let $X$ be a pointed separable metric space and let $\mathcal{U}$ and $\mathcal{V}$ be
subsets of ${\rm Lip}_0(X)$. Then $\mathcal{U}$ and $\mathcal{V}$ are Bernoulli multi-utilities
for the same preorder on $\Delta_1(X)$ if and only if they generate the same {\rm KR}-weak*
closed convex cone.
\end{theo}

\begin{proof}
Denote the KR-weak* closed convex cone generated by any $\mathcal{W} \subseteq {\rm Lip}_0(X)$
by $\mathcal{C}(\mathcal{W})$. As we noted above, any $\mathcal{W} \subseteq {\rm Lip}_0(X)$
induces the same preorder as $\mathcal{C}(\mathcal{W})$. So if $\mathcal{C}(\mathcal{U}) =
\mathcal{C}(\mathcal{V})$, it must be the case that $\mathcal{U}$ and $\mathcal{V}$ are
Bernoulli multi-utilities for the same preorder. Conversely, suppose $\mathcal{C}(\mathcal{U})
\neq \mathcal{C}(\mathcal{V})$ and let $\succsim_{\mathcal{U}}$ and $\succsim_{\mathcal{V}}$
be the induced preorders. Without loss of generality, suppose $\mathcal{C}(\mathcal{U})
\setminus \mathcal{C}(\mathcal{V})$ is nonempty, and pick any $u_0$ in this set. Then by the
separating hyperplane theorem, there exists a $\mu \in {\rm KR}(X)$ such that the map
$\Phi: u \mapsto \int u\, {\rm d}\mu$ on ${\rm Lip}_0(X)$ satisfies $\Phi(u) \geq 0 > \Phi(u_0)$ for
every $u \in \mathcal{C}(\mathcal{V})$. Writing $\mu = \alpha(p - q)$ as in Proposition
\ref{krprop}, this yields that $p \succsim_{\mathcal{V}} q$ but not $p \succsim_{\mathcal{U}} q$.
Thus $\mathcal{C}(\mathcal{U}) \neq \mathcal{C}(\mathcal{V})$ implies that
$\mathcal{U}$ and $\mathcal{V}$ induce different preorders.
\end{proof}

\subsection{Proper representation of Lipschitz affine preorders}
Let $X$ be a separable metric space and let $\succsim$ be a preorder on a nonempty convex
subset $S$ of $\Delta(X)$. We say that a collection $\mathcal{U}$ of Borel measurable functions
on $X$ is {\it strictly $\succsim$-increasing} if
\begin{equation*}
p \succ q\qquad\mbox{implies}\qquad
\int_X u\, {\rm d}p > \int_X u\, {\rm d}q\quad\mbox{for each }u \in \mathcal{U}.
\end{equation*}
In turn, we say that $\mathcal{U}$ is a {\it proper Bernoulli multi-utility} for $\succsim$ if it
is strictly $\succsim$-increasing and (\ref{multiorder2}) holds. Such a representation of
$\succsim$ is of interest because this way we can generate $\succsim$-maximal elements in
any given set simply by maximizing the expectations of the elements of $\mathcal{U}$. In other
words, a proper expected multi-utility representation for $\succsim$ allows us to maximize
$\succsim$ by the so-called scalarization method. (See Section 5 below.)

Proper expected multi-utility representations were studied in detail by Evren \cite{evren14}.
One of the findings of \cite{evren14} is that a continuous affine preorder on $\Delta(X)$ may
fail to admit a proper Bernoulli multi-utility even when $X$ is compact. This is primarily due
to the fact that there is no way of guaranteeing the compactness of the Bernoulli
multi-utilities one finds in the expected multi-utility theorem of \cite{dmo} relative to the
sup norm. In contrast, Corollary \ref{convexmulti} shows that a Lipschitz affine preorder on
$\Delta_1(X)$ always admits a weak* compact Lipschitz Bernoulli multi-utility (for any
separable $X$). We next use this fact to prove that every Lipschitz affine preorder admits a
proper expected multi-utility representation.

\begin{theo}\label{proper}
Let $X$ be a separable pointed metric space and let $\succsim$ be a preorder on $\Delta_1(X)$. Then
$\succsim$ is Lipschitz and affine if and only if it admits a convex proper Lipschitz Bernoulli multi-utility
$\mathcal{V} \subseteq {\rm Lip}_0(X)$.
\end{theo}

\begin{proof}
The reverse direction follows from Proposition \ref{liputilities}. To prove the forward direction,
note first that the convex hull of any proper Bernoulli multi-utility for $\succsim$ is again a proper
Bernoulli multi-utility for $\succsim$. So it will suffice to find a proper Lipschitz Bernoulli multi-utility
$\mathcal{V} \subseteq {\rm Lip}_0(X)$, without the convexity requirement.

By Corollary \ref{convexmulti}, there is a weak* compact and convex
Bernoulli multi-utility $\mathcal{U}$ contained in the closed unit ball of ${\rm Lip}_0(X)$.
Since ${\rm Lip}_0(X)$ has a separable predual, the restriction of the weak* topology to its unit ball,
and therefore also to $\mathcal{U}$, is separable. So let $(u_n)$ be a weak* dense sequence
in $\mathcal{U}$ and define $v:= \sum_{n \geq 1} 2^{-n}u_n$.

Take any $p, q \in \Delta_1(X)$ with $p \succ q$. As $\mathcal{U}$ is a Bernoulli multi-utility for
$\succsim$, we must have $\int u_n\, {\rm d}(p - q) \geq 0$ for all $n$ and
$\int u_{n_0}\, {\rm d}(p - q) > 0$ for some $n_0$. Thus $\int v\, {\rm d}(p - q) > 0$.

To complete the proof, set $\mathcal{V} := \{u + \frac{1}{n} v: u \in \mathcal{U}$ and $n \in \mathbb{N}\}$.
It is clear from the preceding that $\mathcal{V}$ is strictly $\succsim$-increasing, and it is also clear
that $p \succsim q$ implies $\int u\, {\rm d}(p - q) \geq 0$ for every $u \in \mathcal{V}$. To complete
the proof that $\mathcal{V}$ is a Bernoulli multi-utility for $\succsim$, suppose $q \succsim p$ fails.
Then we must have $\int u\, {\rm d}(p - q) > 0$ for some $u \in \mathcal{U}$, and then for sufficiently
large $n$ we will also have $\int (u + \frac{1}{n}v)\, {\rm d}(p - q) > 0$.
\end{proof}

\subsection{An impossibility theorem}
We have seen above that one may represent a Lipschitz affine preorder $\succsim$ on $\Delta_1(X)$
as in the expected multi-utility theorem by a family of Lipschitz functions that is either weak* compact
or strictly $\succsim$-increasing. A natural question is if we can achieve both of these properties in
the representation. Unfortunately, this is possible only in extreme cases, due to the following gem of
topological order theory.
\bigskip

\noindent {\bf Schmeidler's Theorem \cite{schmeidler}.}
{\it Let $E$ be a connected topological space and let $\succsim$ be a continuous preorder on $E$
with $\succ \neq \emptyset$. If $\{x \in E: x \succ e\}$ and $\{x \in E: e \succ x\}$ are open for
every $e \in E$, then $\succsim$ is total.}
\bigskip

We now show that any preorder on $\Delta_1(X)$ that admits a strictly $\succsim$-increasing and
weak* compact Bernoulli multi-utility is either total or unable to render a strict ranking between any
two lotteries.

\begin{prop}\label{nogo}
Let $X$ be a pointed separable metric space and let $\succsim$ be a preorder on $\Delta_1(X)$
with $\succ \neq \emptyset$. If $\succsim$ admits a Lipschitz Bernoulli multi-utility
$\mathcal{U} \subset {\rm Lip}_0(X)$ which is
both strictly $\succsim$-increasing and weak* compact, it must be total.
\end{prop}

\begin{proof}
Suppose there is a strictly $\succsim$-increasing and weak* compact set $\mathcal{U} \subset
{\rm Lip}_0(X)$ which satisfies (\ref{multiorder2}). Then $\succsim$ is $W_1$-continuous. Fix an
arbitrary $q \in \Delta_1(X)$. In view of Schmeidler's theorem, it is enough to show that
$\{p: p \succ q\}$ and $\{p: q \succ p\}$ are open subsets of $\Delta_1(X)$. We do this for the
first of these sets; reversing the partial order then implies the same for the second one.

Suppose $p \succ q$. Then $\int u\, {\rm d}p > \int u\, {\rm d}q$ for every $u \in \mathcal{U}$.
Weak* compactness of $\mathcal{U}$ then implies that the differences 
$\int u\, {\rm d}p - \int u\, {\rm d}q$ are bounded away from zero, i.e., there exists an $\epsilon > 0$
with
\begin{equation*}
\int_X u\, {\rm d}(p - q) > \epsilon
\end{equation*}
for every $u \in \mathcal{U}$. Letting $M = \sup\{L(u): u \in \mathcal{U}\}$, it then follows that
$\int u\, {\rm d}(p' - q) > 0$ for every $p' \in {\rm ball}_{\epsilon/M}(p)$ and every $u \in \mathcal{U}$.
Thus ${\rm ball}_{\epsilon/M}(p) \subseteq \{p': p' \succ q\}$, showing that the latter set is open.
\end{proof}

\subsection{Lipschitz preorders of Levin type}
Let $E$ be a metric space. In view of \cite{levin}, we say that a preorder $\succsim$
on $E$ is {\it Lipschitz in the sense of Levin} if for every $x, y \in E$ for which $y \succsim x$ fails,
there exists a $\succsim$-increasing Lipschitz function $U$ on $E$ with $U(x) > U(y)$. Equivalently,
$\succsim$ is Lipschitz in the sense of Levin if, choosing an arbitrary base point for $E$, there exists
a family $\mathcal{U} \subseteq {\rm Lip}_0(E)$ such that
\begin{equation*}
x \geq y\qquad\mbox{if and only if}\qquad U(x) \geq U(y)\quad
\mbox{for every }U \in \mathcal{U}.
\end{equation*}

Any such preorder is continuous, that is, it is a closed subset of $E \times E$. It is shown in
\cite{levin} that if $E$ is separable, then any total preorder $\succsim$ on $E$ is Lipschitz in the
sense of Levin if and only if there exists a Lipschitz
$U: E \to \mathbb{R}$ such that $x \succsim y$ iff $U(x) \geq U(y)$.

Levin's definition deftly captures an important general notion, but it has the arguable drawback
of being formulated in terms of the utility functions one would use to represent the given
preorder. In decision theory, one considers properties that are directly expressible in terms of
a preorder (modeling one's preference relation) as {\it behavioral}. After all, $\succsim$ is, at least
in principle, observable in terms of the choices of an individual in pairwise choice
situations.\footnote{Of course, a continuity type axiom can never be tested in real life. However,
imposing such a property on a preference relation nevertheless allows for an obvious behavioral
interpretation.} In contrast, the idea of utility is a purely mathematical construct. It is very useful,
for it not only highlights the inner structure of a preorder, but it also makes working with that
preorder easier in practice. Utility functions are, however, not observable (even in finitistic settings),
and provide behavioral information about one's preferences only indirectly. It is therefore difficult
to identify the behavioral content of Levin's Lipschitz property.

Our Lipschitz property (Definition \ref{lipdef}) is behavioral, but it is less general than Levin's,
as it only applies to preorders on the space $\Delta_1(X)$ for some separable metric space $X$.
It is then natural to inquire into the relation between our property and Levin's for affine preorders,
the principal objects we study in this paper, on such a domain. An easy observation is that
if an affine preorder is Lipschitz in our sense then Theorem \ref{mainthm} implies that it is Lipschitz
in Levin's sense. Indeed, the involved (von Neumann-Morgenstern) utility functions
$U: p \mapsto \int u\, {\rm d}p$ for $u \in \mathcal{U}$, are Lipschitz for the $W_1$ metric
on $\Delta_1(X)$, with the same Lipschitz number as $u$, because
\begin{equation*}
|U(p) - U(q)| = \left|\int_X u\, {\rm d}(p - q)\right| \leq L(u)\cdot W_1(p,q)
\end{equation*}
for every $p, q \in \Delta_1(X)$.

We will next prove the converse of this observation, which shows that Levin's Lipschitz property
is equivalent to ours when applied to affine preorders on $\Delta_1(X)$. This is harder because our
von Neumann-Morgenstern utility functions are affine on $\Delta_1(X)$, whereas Levin's need not be.
Intuitively, however, we can linearize any non-affine Lipschitz function by replacing it with its
derivative at any point of differentiability. This idea is not immediately applicable to the case where
$\Delta_1(X)$ is infinite dimensional, but it can be made to work via finite dimensional approximation.

\begin{lemma}\label{levin1}
Let $S$ be a closed convex $m$-dimensional subset of $\mathbb{R}^m$ equipped with an affine preorder
$\succsim$. Suppose $f: S \to \mathbb{R}$ is 1-Lipschitz and $\succsim$-increasing, and let $x, y \in S$ and
$\epsilon > 0$. Then there is a point $z \in S$ at which $f$ is differentiable and its derivative $g$ satisfies
\begin{equation*}
f(x) - f(y) - \epsilon < g(x) - g(y) < f(x) - f(y) + \epsilon.
\end{equation*}
Moreover, the restriction of $g$ to $S$ is also $\succsim$-increasing.
\end{lemma}

\begin{proof}
The first part is a standard consequence of Rademacher's theorem; see, e.g., the proof of
\cite[Theorem 1.41]{weaver}. For the second part, fix $x', y' \in S$ with $x' \succsim y'$. By translating
$S$, we may assume $z = 0$, and by shifting $f$ by a constant we may assume $f(0) = 0$. Now
for any $n \in \mathbb{N}$ we have $\frac{1}{n}x' \succsim \frac{1}{n}y'$ by affinity, so that
$nf(\frac{1}{n}x') \geq nf(\frac{1}{n}y')$ (since $f$ is $\succsim$-increasing). Taking $n \to \infty$, this
yields $g(x') \geq g(y')$. (Since $f(0) = 0$, $\frac{f(x'/n)}{1/n} \to g(x')$ directly from the definition of the
derivative.) This shows that $g$ is $\succsim$-increasing.
\end{proof}

\begin{lemma}\label{levin2}
With the same hypotheses as in Lemma \ref{levin1}, there is an $\succsim$-increasing affine 1-Lipschitz
function $h: S \to \mathbb{R}$ which satisfies $h(x) = f(x)$ and $h(y) = f(y)$.
\end{lemma}

\begin{proof}
For each $n \in \mathbb{N}$, taking $\epsilon = \frac{1}{n}$ in Lemma \ref{levin1}, we can find a
linear map $g_n: \mathbb{R}^m \to \mathbb{R}$ which is 1-Lipschitz and $\succsim$-increasing on $S$ and
satisfies $f(x) - f(y) - \frac{1}{n} < g_n(x) - g_n(y) < f(x) - f(y) + \frac{1}{n}$. Let $h_n = g_n|_S + f(x) - g_n(x)$;
then $h_n$ is a 1-Lipschitz, $\succsim$-increasing, affine function on $S$ which satisfies $h_n(x) = f(x)$ and
$|h_n(y) - f(y)| < \frac{1}{n}$. Taking a cluster point of the sequence $(h_n)$ yields the desired conclusion.
\end{proof}

\begin{theo}\label{levinthm}
Let $X$ be a separable metric space and let $\succsim$ be an affine preorder on $\Delta_1(X)$.
Then $\succsim$ is Lipschitz if and only if it is Lipschitz in the sense of Levin.
\end{theo}

\begin{proof}
The forward direction follows immediately from Theorem \ref{mainthm}. For the reverse direction,
suppose $\succsim$ is Lipschitz in the sense of Levin. Fix $p, q \in \Delta_1(X)$ and suppose
$q \succsim p$ fails. Then Levin's condition provides us with a 1-Lipschitz $\succsim$-increasing function
$f: \Delta_1(X) \to \mathbb{R}$ satisfying $f(p) > f(q)$. We require a Lipschitz function $u = u_{p,q}$ on $X$
satisfying $\int u\, {\rm d}p > \int u\, {\rm d}q$, and $\int u\, {\rm d}p' \geq \int u\, {\rm d}q'$
whenever $p' \succsim q'$; Proposition \ref{liputilities} can then be applied to show that $\succsim$ is
Lipschitz using the set $\mathcal{U} = \{u_{p,q}: q \succsim p$ fails$\}$.

Fix a base point $e$ of $X$.
For each finite subset $F$ of $\Delta_1(X)$ which contains $p$ and $q$, let $S_F$ be its convex hull.
This is a finite dimensional closed convex subset, and the restriction of $f$ to $S_F$ is 1-Lipschitz and
$\succsim$-increasing, so Lemma \ref{levin2}
yields an affine 1-Lipschitz function $h_F: S_F \to \mathbb{R}$ which is $\succsim$-increasing and satisfies
$h_F(p) = f(p)$ and $h_F(q) = f(q)$. Then define $k_F$ on $\{\alpha(p' - q'): \alpha \geq 0$ and
$p', q' \in S_F\} \subseteq {\rm KR}(X)$ by $k_F(\alpha(p' - q')) = \alpha(h_F(p') - h_F(q'))$. This is
a bounded linear functional of norm at most 1 on a finite dimensional subspace of ${\rm KR}(X)$.
By the Hahn-Banach theorem, it extends to a bounded linear functional of norm at most 1 on all of
${\rm KR}(X)$, and hence is given by integration against a 1-Lipschitz function $u_F \in {\rm Lip}_0(X)$
(Theorem \ref{duality}). Ordering the finite subsets $F$ by inclusion, we can pass to a subnet of $(u_F)$
which converges weak* to some $u \in {\rm Lip}_0(X)$. This will be a 1-Lipschitz function on $X$ with
the property that
\begin{equation*}
h_F(p') - h_F(\delta_e) = k_F(p' - \delta_e)  = \int_X u_F\, {\rm d}(p' - \delta_e)
= \int_X u_F\, {\rm d}p' \to \int_X u\, {\rm d}p'
\end{equation*}
for every $p' \in \Delta_1(X)$. Thus, if $p' \succsim q'$, then $h_F(p') \geq h_F(q')$ (for all $F$ containing
$p'$ and $q'$) implies $\int u\, {\rm d}p' \geq \int u\, {\rm d}q'$, and also
$\int u\, {\rm d}p = f(p) - h_F(\delta_e) > f(q) - h_F(\delta_e) = \int u\, {\rm d}q$, so $u$ has the requires properties.
\end{proof}

\section{On the maximization of Lipschitz affine preorders}
Let $X$ be a separable metric space and let $\succsim$ be a Lipschitz affine predorder on $\Delta_1(X)$.
We presently have two approaches to representing $\succsim$ in terms of Lipschitz Bernoulli
multi-utilities, one weak* compact and convex (Corollary \ref{convexmulti}) and the other strictly
$\succsim$-increasing and convex (Theorem \ref{proper}). Each of these approaches
has its own advantages. For instance,
in the second case, maximizing any of the Bernoulli utility functions on a given set $P \subseteq \Delta_1(X)$
yields a $\succsim$-maximal lottery in $P$. Thus, putting together the maxima of all the involved Bernoulli
utilities on $P$ yields a lower estimate for ${\rm MAX}(P, \succsim)$, the set of all $\succsim$-maximal
lotteries in $P$ (i.e., the set of all $p \in P$ such that no $q \in P$ satisfies $q \succ p$). On the
other hand, carrying out the same procedure with Bernoulli utilities of the first type  yields an upper
estimate for ${\rm MAX}(P, \succsim)$. That is, every $\succsim$-maximal lottery in $P$ maximizes at least
one of the Bernoulli utilities at hand in that case. This is the content of the next result.

\begin{prop}
Let $X$ be a pointed separable metric space and let $\succsim$ be a preorder on $\Delta_1(X)$. Let
$\mathcal{U}, \mathcal{V} \subseteq {\rm Lip}_0(X)$ be Bernoulli multi-utilities for $\succsim$, and
suppose both are convex, with $\mathcal{U}$ weak* compact and $\mathcal{V}$ strictly
$\succsim$-increasing. Then
\begin{equation*}
\bigcup_{v \in \mathcal{V}} \underset{p \in P}{\rm arg\, max}
\int_X v\, {\rm d}p \quad \subseteq \quad {\rm MAX}(P, \succsim)
\quad\subseteq\quad \bigcup_{u \in \mathcal{U}}
\underset{p \in P}{\rm arg\, max} \int_X u\, {\rm d}p
\end{equation*}
for any $P \subseteq \Delta_1(X)$.
\end{prop}

\begin{proof}
Fix $P \subseteq \Delta_1(X)$. The first containment follows readily from the fact that every
member of $\mathcal{V}$ is strictly $\succsim$-increasing. To prove the second containment,
take any $\sigma \in \Delta(\mathcal{U})$, i.e., any Borel probability measure on $\mathcal{U}$
equipped with the relative weak* topology. By \cite[Proposition 1.1]{phelps2}, there is a (unique)
$u \in \mathcal{U}$ such that $F(u) = \int F(v)\, \sigma({\rm d}v)$ for every weak* continuous
linear functional $F$ on ${\rm Lip}_0(X)$. But $w \mapsto \int w\, {\rm d}p$ is indeed such a
functional on ${\rm Lip}_0(X)$ for every $p \in \Delta_1(X)$. Consequently,
\begin{equation*}
\int_X u\, {\rm d}p = \int_{\mathcal{U}} \int_X v\, {\rm d}p\, \sigma({\rm d}v)
\end{equation*}
for every $p \in \Delta_1(X)$. We thus conclude that for every $\sigma \in \Delta(\mathcal{U})$,
there is a $u \in \mathcal{U}$ such that
\begin{equation*}
\underset{p \in P}{\rm arg\, max} \int_X u\, {\rm d}p =
\underset{p \in P}{\rm arg\, max} \int_{\mathcal{U}} \int_X v\, {\rm d}p\, \sigma({\rm d}v)
\end{equation*}
for every $p \in P$. This proves the $\supseteq$ part of the equation
\begin{equation}\label{argmax}
\bigcup_{u \in \mathcal{U}} \underset{p \in P}{\rm arg\, max} \int_X u\, {\rm d}p =
\bigcup_{\sigma \in \Delta(\mathcal{U})} \underset{p \in P}{\rm arg\, max}
\int_{\mathcal{U}} \int_X v\, {\rm d}p\, \sigma({\rm d}v)
\end{equation}
whose $\subseteq$ part is trivial.

Now take any $p^* \in {\rm MAX}(P, \succsim)$. We will complete the proof by showing that $p^*$
belongs to the right hand side of (\ref{argmax}). To this end, for any $p \in \Delta_1(X)$, define
$\zeta_p: \mathcal{U} \to \mathbb{R}$ by $\zeta_p(u) := \int u\, {\rm d}(p - p^*)$. It is plain that
$\{\zeta_p: p \in P\}$ is a convex subset of $C(\mathcal{U})$. If there were a $p \in P$ such that
$\zeta_p(u) > 0$ for each $u \in \mathcal{U}$, we would have $p \succ p^*$, contradicting
$\succsim$-maximality of $p^*$ in $P$. Thus $\{\zeta_p: p \in P\}$ is disjoint from the open convex
cone $\{f \in C(\mathcal{U}): f > 0\}$. So by the separating hyperplane theorem, there exists a positive
linear functional $F \in C(\mathcal{U})^*$ such that $F(\zeta_p) \leq 0$ for every $p \in P$,
i.e., there is a $\sigma \in \Delta(\mathcal{U})$ such that
\begin{equation*}
\int_{\mathcal{U}} \zeta_p(u)\, \sigma({\rm d}v) \leq 0
\end{equation*}
for every $p \in P$. Thus
\begin{equation*}
p^* \in \underset{p \in P}{\rm arg\, max} \int_{\mathcal{U}}\int_X v\, {\rm d}p\, \sigma({\rm d}v),
\end{equation*}
as desired.
\end{proof}

This result suggests that one should look for a Lipschitz Bernoulli multi-utility for a preorder
$\succsim$ which is both weak* compact and strictly $\succsim$-increasing. Unfortunately,
as we have seen in Proposition \ref{nogo}, this is not possible unless $\succsim$ is total
(in which case there is no need for a scalarization exercise).

\section{Application: Representation of the affine core of a preorder}

\subsection{Affine core of a preorder}
Let $\succsim$ be a preorder on a nonempty convex subset of a linear space. By the {\it affine core}
of $\succsim$, we mean the largest affine subrelation of $\succsim$, and we denote this subrelation
as $\succsim_{\rm aff}$. It is plain that $\succsim$ is affine iff $\succsim\, =\, \succsim_{\rm aff}$. This
concept is used widely in boundedly rational decision theory when studying preference relations over
lotteries, or state-dependent lotteries (see Section 7), that may fail to be affine. In fact, it
was first introduced in the latter context by Ghirardato, Maccheroni, and Marinacci \cite{gmm}, who
referred to it as the ``revealed unambiguous preference.'' In the context of risk, a variant (which is
really a superrelation of $\succsim_{\rm aff}$) is often referred to as the ``linear core'' of $\succsim$. This 
concept was studied in depth in \cite{cv, cvdo}, among others.\footnote{Put precisely, these authors define
the {\it linear core} of a preorder on a nonempty convex subset $S$ of $\Delta(X)$ as the preorder $\unrhd$
on $S$ with $p \unrhd q$ iff $p \oplus_\lambda r \succsim q \oplus_\lambda r$ for all $r \in S$ and
$\lambda \in (0,1]$. Then $\unrhd$ is weakly affine, but it need not be affine. As we shall
see shortly, however, $\unrhd$ becomes affine if $\succsim$ is suitably continuous, and the difference
between the two notions disappears.}

In general, the affine core of $\succsim$ need not exist. But we shall shortly prove that it exists as an
affine preorder under fairly general conditions.

\begin{lemma}\label{walemma}
Any weakly affine, continuous preorder on a nonempty convex subset of a topological vector space
is affine.
\end{lemma}

\begin{proof}
Let $\succsim$ be a weakly affine, continuous preorder on a nonempty convex subset $S$ of a topological
vector space and suppose
\begin{equation}\label{weakaffine}
p \oplus_\lambda r \succsim q\oplus_\lambda r
\end{equation}
for some $p, q, r \in S$ and $\lambda \in [0,1)$; we must show that $p \succsim q$. Put
$A := \{\beta \in [0,1): p \oplus_\beta r \succsim q\oplus_\beta r\}$ and let $\alpha := \inf A$.
Then there is a sequence $(\alpha_n)$ in $A$ which converges to $\alpha$,
and thus $p \oplus_{\alpha_n} r \to p \oplus_\alpha r$ and $q \oplus_{\alpha_n} r \to q \oplus_\alpha r$.
So continuity of $\succsim$ implies that $\alpha \in A$. Now set $\beta = \frac{1 - \alpha}{2 - \alpha}$
and use (\ref{weakaffine}) twice to find
\begin{equation*}
\left(p\underset{\alpha}{\oplus} r\right) \underset{\beta}{\oplus} p \succsim
\left(q\underset{\alpha}{\oplus} r\right) \underset{\beta}{\oplus} p =
\left(p\underset{\alpha}{\oplus} r\right) \underset{\beta}{\oplus} q \succsim
\left(q\underset{\alpha}{\oplus} r\right) \underset{\beta}{\oplus} q,
\end{equation*}
so that $p \oplus_{\alpha/(2 - \alpha)} r \succsim q \oplus_{\alpha/(2 - \alpha)} r$. By the definition of $\alpha$,
this shows that $\alpha \leq \frac{\alpha}{2 - \alpha}$, that is, $\alpha \leq \alpha^2$. Since $\alpha \in [0,1)$, we
obtain $\alpha = 0$, which means that $p \succsim q$, as desired.
\end{proof}

The continuity hypothesis in Lemma \ref{walemma} is very mild because any topology compatible with the
vector space structure can be used; in particular, we could take the coarsest possible compatible topology, namely
the weak topology determined by the family of all linear functionals. All we need is that $\alpha_n \to \alpha$
and $p + \alpha_n(r - p) \succsim q + \alpha_n(r - q)$ for all $n$ imply $p + \alpha(r - p) \succsim q + \alpha(r - q)$.

\begin{prop}\label{affprop}
Let $X$ be a separable metric space and let $\succsim$ be a $W_1$-continuous preorder on
$\Delta_1(X)$. Then $\succsim_{\rm aff}$ exists as an affine preorder on $\Delta_1(X)$, and it is
characterized as:
\begin{equation*}
p \succsim_{\rm aff} q \qquad\mbox{iff}\qquad
p \underset{\lambda}{\oplus} r \succsim q \underset{\lambda}{\oplus} r\quad\mbox{for every }
r \in \Delta_1(X)\mbox{ and }\lambda \in [0,1].
\end{equation*}
\end{prop}

\begin{proof}
Define a binary relation $\unrhd$ on $\Delta_1(X)$ by setting $p\unrhd q$ iff
$p \oplus_\lambda r \succsim q \oplus_\lambda r$ for every $r \in \Delta_1(X)$ and $\lambda \in [0,1]$.
Obviously, $\unrhd$ is a preorder on $\Delta_1(X)$ with $\unrhd \subseteq\, \succsim$. It is also plain
that if $\unrhd'$ is an affine subrelation of $\succsim$, then $\unrhd' \subseteq \unrhd$. It remains to
show that $\unrhd$ is affine. To this end, take any $p, q, r \in \Delta_1(X)$ with $p \unrhd q$ and any
$\lambda \in [0,1]$. We will first prove that $p\oplus_\lambda r \unrhd q \oplus_\lambda r$, that is,
\begin{equation}\label{affeq}
\left(p \underset{\lambda}{\oplus} r\right) \underset{\alpha}{\oplus} s \succsim
\left(q \underset{\lambda}{\oplus} r\right) \underset{\alpha}{\oplus} s
\end{equation}
for every $s \in \Delta_1(X)$ and $\alpha \in [0,1]$. Fix $s$ and $\alpha$. Now put $\beta =
\alpha + (1 - \alpha)\lambda$ and $s' := r \oplus_{\alpha/\beta} s$, and note that
$(p \oplus_\lambda r) \oplus_\alpha s = p \oplus_\beta s'$ and $(q \oplus_\lambda r) \oplus_\alpha s
= q \oplus_\beta s'$. Therefore, as $p \unrhd q$, it follows from the definition of
$\unrhd$ that (\ref{affeq}) holds. Conclusion: $\unrhd$ is weakly affine.

We next claim that $\unrhd$ is $W_1$-continuous. Indeed, let $(p_n)$ and $(q_n)$ be two sequences
in $\Delta_1(X)$ such that $p_n \unrhd q_n$ for all $n \in \mathbb{N}$, and assume that
$W_1(p_n,p) \to 0$ and $W_1(q_n,q) \to 0$ for some $p,q \in \Delta_1(X)$. Take any
$r \in \Delta_1(X)$ and any $\lambda \in [0,1]$, and note that $W_1(p_n \oplus_\lambda r,
p \oplus_\lambda r) = (1 - \lambda)W_1(p_n,p) \to 0$, and similarly $W_1(q_n \oplus_\lambda r,
q \oplus_\lambda r) \to 0$. Since $p_n \oplus_\lambda r \succsim
q_n \oplus_\lambda r$ for each $n$ (by the definition of $\unrhd$), and $\succsim$ is
$W_1$-continuous, we obtain $p \oplus_\lambda r \succsim q \oplus_\lambda r$. Since
$r$ and $\lambda$ were arbitrary, we conclude that $p \unrhd q$. Thus $\unrhd$ is
$W_1$-continuous. We have proven that $\unrhd$ is $W_1$-continuous and weakly affine;
by Lemma \ref{walemma}, therefore, it is affine.
\end{proof}

\subsection{Functional representation of the affine core}
We can now prove an expected multi-utility representation for the affine core
of any preorder on $\Delta_1(X)$ that is Lipschitz in the sense of Levin.

\begin{theo}
Let $X$ be a separable metric space and $\succsim$ a preorder on $\Delta_1(X)$ which is Lipschitz
in the sense of Levin. Then there is a family $\mathcal{U}$ of Lipschitz functions on $\Delta_1(X)$
such that
\begin{equation*}
p \succsim_{\rm aff} q \qquad\mbox{iff}\qquad
\int_X u\, {\rm d}p \geq \int_X u\, {\rm d}q\quad\mbox{for every }u \in \mathcal{U}.
\end{equation*}
\end{theo}

\begin{proof}
We claim that $\succsim_{\rm aff}$ is Lipschitz in the sense of Levin. To prove this, take any
$p, q \in \Delta_1(X)$ such that $q \succsim_{\rm aff} p$ fails. Then, by Proposition \ref{affprop},
$q \oplus_\lambda r \succsim p \oplus_\lambda r$ fails for some $r \in \Delta_1(X)$ and
$\lambda \in [0,1]$. As $\succsim$ is Lipschitz in the sense of Levin, then, there is a
$\succsim$-increasing 1-Lipschitz function $U$ on $\Delta_1(X)$ with $U(p\oplus_\lambda r) >
U(q \oplus_\lambda r)$. Now define $V: \Delta_1(X) \to \mathbb{R}$ by $V(w) := U(w \oplus_\lambda r)$.
Clearly, $V$ is 1-Lipschitz and $V(p) > V(q)$. Moreover, if $w \succsim_{\rm aff} w'$, then by Proposition
\ref{affprop}, $w \oplus_\lambda r \succsim w' \oplus_\lambda r$, whence $V(w) = U(w \oplus_\lambda r)
\geq U(w' \oplus_\lambda r) = V(w')$ because $U$ is $\succsim$-increasing. Conclusion: $V$ is a
$\succsim_{\rm aff}$-increasing 1-Lipschitz function on $\Delta_1(X)$ with $V(p) > V(q)$. We have shown
that $\succsim_{\rm aff}$ is Lipschitz in the sense of Levin. As $\succsim_{\rm aff}$ is an affine preorder
(Proposition \ref{affprop}), therefore, by Theorem \ref{levinthm}, $\succsim_{\rm aff}$ is Lipschitz.
The present result thus obtains by applying Theorem \ref{mainthm}.
\end{proof}

\section{Application: A single-prior expected multi-utility theorem}

\subsection{The Anscombe-Aumann framework}
The analysis of preferences over risky prospects models uncertainty about possible outcomes in terms
of objective probabilities (as in games of chance). However, in many situations of interest (such as betting
on teams, or choosing a financial portfolio), the relevant probabilities are subjective, because two rational
individuals may disagree on the likelihood of various states of nature, and hence the final outcomes. This
necessitates specifying a state space $\Omega$, and then studying preferences over state-dependent
prospects, often called {\it Savagean acts}, which are $X$-valued maps on $\Omega$. (Choosing an act
$f$ returns the outcome $f(\omega)$ if the state that is realized after the act is chosen is $\omega \in \Omega$.)
This leads to the so-called Savagean theory of decision-making under uncertainty. The highly influential work
of Anscombe and Aumann \cite{aa} showed that one would better tie this model to the classical von
Neumann-Morgenstern theory by instead considering $\Delta(X)$-valued acts on $\Omega$. The resulting
model, which is the workhorse of modern decision theory under uncertainty, is described as follows.

We designate a nonempty finite set $\Omega$ as the set of states of nature, and a separable metric
space $X$ as a prize space. An {\it Anscombe-Aumann act}, or simply an {\it act}, is a function from $\Omega$
into $\Delta_1(X)$; such a map models a state-contingent plan whose returns are lotteries (with objectively
specified probabilities). We endow the set of all acts $\Delta_1(X)^\Omega$ with the
product metric. Write $W_1(f,g)$ for $\sum_{\omega \in \Omega} W_1(f(\omega), g(\omega))$, for
any acts $f$ and $g$, and for any $\lambda \in [0,1]$ and any acts $f$ and $g$ define
$f \oplus_\lambda g$ to be the act $\omega \mapsto f(\omega) \oplus_\lambda g(\omega)$.

A preorder on $\Delta_1(X)^\Omega$ is interpreted as the preference relation of a rational
individual who does not know the state of nature before choosing an action. We say that such a preorder
$\succsim$ is {\it affine} provided that $f \succsim g$ iff $f \oplus_\lambda h \succsim
g \oplus_\lambda h$, for any $h \in \Delta_1(X)^\Omega$ and $\lambda \in [0,1)$, and
{\it Lipschitz} if for every $f,g \in \Delta_1(X)^\Omega$ such that $g \succsim f$ fails, there
exists $K > 0$ such that $g \oplus_\lambda g' \succsim f \oplus_\lambda f'$ fails for every
$f', g' \in \Delta_1(X)^\Omega$ and every
$\lambda$ in the interval $\left[0, \frac{K}{K + W_1(f',g')}\right)$. The following is a straightforward
generalization of Theorem \ref{mainthm}.

\begin{theo}\label{spthm}
Let $\Omega$ be a finite set and let $X$ be a pointed separable metric space. Then a preorder $\succsim$
on $\Delta_1(X)^\Omega$ is affine and Lipschitz if and only if there exists a family
$\mathbb{U} \subseteq {\rm Lip}_0(X)^\Omega$ such that
\begin{equation}\label{speqn}
f \succsim g\qquad\mbox{iff}\qquad \sum_{\omega \in \Omega} \int_X u_\omega\, {\rm d}f(\omega)
\geq \sum_{\omega \in \Omega} \int_X u_\omega\, {\rm d}g(\omega)\quad\mbox{for every }
u \in \mathbb{U}.
\end{equation}
\end{theo}

Here we write $u_\omega \in {\rm Lip}_0(X)$ for the $\omega$th coordinate of $u \in {\rm Lip}_0(X)^\Omega$.
This result is proved in exactly the same way we proved Theorem 4.4, but this time making use of the
duality $\left(\bigoplus_{\omega \in \Omega} {\rm KR}(X)\right)^* \cong \bigoplus_{\omega \in \Omega}
{\rm Lip}_0(X)$ (see the end of Section 2.2). We omit the details.

\subsection{Local probabilistic sophistication}
For any $\alpha \in \Delta(\Omega)$ and any act $f$, we write $f^\alpha$ for the constant act that maps
any state $\nu \in \Omega$ to the probability measure $\sum_{\omega \in \Omega} \alpha\{\omega\}
f(\omega)$. Adapting the main definition of Machina and Schmeidler \cite{ms}, we say that a preorder
$\succsim$ on $\Delta_1(X)^\Omega$ is {\it probabilistically sophisticated} if there exists an
$\alpha \in \Delta(\Omega)$ such that $f \sim f^\alpha$ for every $f \in \Delta_1(X)^\Omega$.
In words, probabilistic sophistication of $\succsim$ means that the individual has a uniform method
of reducing any act to a lottery, thereby replacing the underlying uncertainty about the states with
risk. (The preferences of such individuals over acts are recovered completely from their preferences
over lotteries.)

Ok, Ortoleva, and Riella \cite{oor} have introduced a much weaker version of probabilistic sophistication
where again one reduces an act to a constant act, but this time not uniformly. Put precisely, we say that
a preorder $\succsim$ on $\Delta_1(X)^\Omega$ is {\it locally probabilistically sophisticated} if for
every $f \in \Delta_1(X)^\Omega$ there exists an $\alpha \in \Delta(\Omega)$ such that
$f \sim f^\alpha$. (This property was called the ``reduction axiom'' in \cite{oor}.) Our final result in
this paper identifies how Theorem \ref{spthm} adjusts for locally probabilistically sophisticated
preferences. In this case, we not only find cardinal multi-utilities that represent one's
preferences over lotteries, but also pin down one's beliefs about the states of nature uniquely,
provided that $\succsim$ is {\it nontrivial} in the sense that it does not render all acts indifferent,
i.e., $\succsim\, \neq \Delta_1(X)^\Omega\times\Delta_1(X)^\Omega$. Such results are
often called ``single prior expected multi-utility theorems'' in the literature on decision making under
uncertainty (where $\mu$ in the representation below is interpreted as the prior beliefs of the individual
about the likelihoods of the states of nature). In particular, \cite{oor} proves an analogous result to
Theorem \ref{finalthm} under the assumptions that $X$ is compact and $\succsim$ is continuous, and
finds a $\mathcal{U}$ that consists of continuous functions. As in Theorem \ref{mainthm}, our contribution
here is to show that replacing the continuity requirement on $\succsim$ by the Lipschitz property allows us
to relax the compactness hypothesis to separability and obtain Lipschitz Bernoulli utilities.

\begin{theo}\label{finalthm}
Let $\Omega$ be a finite set, let $X$ be a pointed separable metric space, and let $\succsim$ be a nontrivial
preorder on $\Delta_1(X)^\Omega$. Then $\succsim$ is locally probabilistically sophisticated,
Lipschitz, and affine if and only if there exists a unique $\mu \in \Delta(\Omega)$ and a family
$\mathcal{U} \subseteq {\rm Lip}_0(X)$ such that
\begin{equation*}
f \succsim g \qquad\mbox{iff}\qquad
\sum_{\omega \in \Omega} \mu\{\omega\} \int_X u\, {\rm d}f(\omega) \geq
\sum_{\omega \in \Omega} \mu\{\omega\} \int_X u\, {\rm d}g(\omega)
\quad\mbox{for every }u \in \mathcal{U}.
\end{equation*}
\end{theo}

\begin{proof}
The reverse implication is easy. (Indeed, in this direction one need not assume $\mu$ is unique, and one
can infer that $\succsim$ is, in fact, probabilistically sophisticated.)
For the forward implication, let $\succsim$ be a locally probabilistically
sophisticated Lipschitz affine preorder on $\Delta_1(X)^\Omega$. By Theorem \ref{spthm}, there
exists a family $\mathbb{U} \subseteq {\rm Lip}_0(X)^\Omega$ such that (\ref{speqn}) holds. We can
assume that every $u \in \mathbb{U}$ is nonzero on some coordinate.

Take an arbitrary $u \in \mathbb{U}$ and fix a coordinate $\omega_u \in \Omega$ such that
$u_{\omega_u} \neq 0$. To simplify notation, define $u_* = u_{\omega_u} \in {\rm Lip}_0(X)$. We claim
that each coordinate $u_\omega$ of $u$ is a nonnegative multiple of $u_*$.
To derive a contradiction, suppose there exists a $\nu \in \Omega$ for which this fails, i.e., $u_*$ and $u_\nu$
are either linearly independent or one is a negative scalar multiple of the other. Then there is an element
of the predual whose pairing with one of these is positive and with the other is negative; recalling Proposition
\ref{krprop} (and Theorem \ref{duality}), it follows that there must exist two lotteries $p, q \in \Delta_1(X)$
such that
\begin{equation}\label{omegau}
\int_X u_*\, {\rm d}(p-q) > 0\qquad\mbox{and}\qquad
\int_X u_\nu\, {\rm d}(q-p) > 0.
\end{equation}
Let $A:= \{\omega \in \Omega: \int u_\omega\, {\rm d}(p - q) > 0\}$ and $B:= \Omega\setminus A$;
both of these sets are nonempty by (\ref{omegau}). Now consider the act $f \in \Delta_1(X)^\Omega$
with $f|_A = p$ and $f|_B = q$. As $\succsim$ is locally probabilistically sophisticated, there is an
$\alpha \in \Delta(\Omega)$ with $f \sim f^\alpha$. But this is impossible because
\begin{eqnarray*}
\sum_{\omega \in \Omega} \int_X u_\omega\, {\rm d}f^\alpha(\omega)
&=& \sum_{\omega \in \Omega}\left(\alpha(A)\int_X u_\omega\, {\rm d}p
+ (1 - \alpha(A))\int_X u_\omega\, {\rm d}q\right)\cr
&<& \sum_{\omega \in \Omega} \int_X u_\omega\, {\rm d}f(\omega),
\end{eqnarray*}
where the inequality is seen by comparing the sums term by term.

By what we have found in the previous paragraph, for every $u \in \mathbb{U}$ there exists a vector
$a \in [0,\infty)^\Omega$ such that $u_\omega = a_\omega u_*$ for all $\omega \in \Omega$.
Let $\mu_u \in \Delta(\Omega)$ be the probability measure with $\mu_u\{\omega\} =
\frac{1}{\theta}a_\omega$ where $\theta = \sum_{\omega \in \Omega} a_\omega$. Note that
\begin{equation}\label{constant}
f\succsim g\qquad\mbox{iff}\qquad \sum_{\omega \in \Omega} \mu_u\{\omega\}
\int_X u_*\, {\rm d}(f(\omega) - g(\omega)) \geq 0\quad\mbox{for every }u \in \mathbb{U}.
\end{equation}

Now take any distinct $u, u' \in \mathbb{U}$. We aim to show that $\mu_u = \mu_{u'}$. To this
end, fix an arbitrary state $\omega_0 \in \Omega$. Find $x,y \in X$ such that $u_*(x) - u_*(y)$ and
$u_*'(x) - u_*'(y)$ are both nonzero\footnote{If there is any point at which both $u_*$ and $u_*'$ are
nonzero, that point and the base point work. Otherwise, find $x \in X$ at which $u_*(x) \neq 0$
(and hence $u_*'(x) = 0$) and $y \in X$ at which $u_*'(y) \neq 0$ (and hence $u_*(y) = 0$); then this
pair of points works.}
and define $f \in \Delta_1(X)^\Omega$ by $f(\omega_0):= \delta_x$ and $f(\omega):= \delta_y$
for any $\omega \in \Omega\setminus\{\omega_0\}$.
Since $\succsim$ is locally probabilistically sophisticated, there is an $\alpha \in \Delta(\Omega)$ with
$f \sim f^\alpha$. But $(f - f^\alpha)(\omega_0) = (1 - \alpha\{\omega_0\})(\delta_x - \delta_y)$ and
$(f - f^\alpha)(\omega) = -\alpha\{\omega_0\}(\delta_x - \delta_y)$ for any $\omega \in \Omega \setminus
\{\omega_0\}$, so (\ref{constant}) entails
\begin{equation*}
\mu_u\{\omega_0\}(1 - \alpha\{\omega_0\})(u_*(x) - u_*(y)) - (1 - \mu_u\{\omega_0\})\alpha\{\omega_0\}
(u_*(x) - u_*(y)) = 0,
\end{equation*}
i.e., $(\mu_u\{\omega_0\} - \alpha\{\omega_0\})(u_*(x) - u_*(y)) = 0$, and therefore $\mu_u\{\omega_0\} =
\alpha\{\omega_0\}$. Applying the same argument to $u'$ yields $\mu_{u'}\{\omega_0\} =
\alpha\{\omega_0\}$, and since $\omega_0$ was arbitrary we conclude that $\mu_u = \mu_{u'}$, as
desired. Combining this with (\ref{constant}) yields the representation claimed in the theorem.

It remains to prove the uniqueness claim. Let us assume that pairs $(\mu,\mathcal{U})$ and
$(\mu', \mathcal{U}')$ both satisfy all the requirements of the theorem. Since $\succsim$ is nontrivial,
there is some nonzero $u \in \mathcal{U}$; fix $x \in X$ with $u(x) \neq 0$. Then $\delta_x \sim \delta_e$
must fail, which means that we must also have $u'(x) \neq 0$ for some $u' \in \mathcal{U}'$. Now the
argument of the previous paragraph, applied to $x$ and $e$, yields that $\mu = \mu'$.
\end{proof}

\end{document}